\documentclass[12pt,reqno]{amsart}
\usepackage{}
\usepackage{mathrsfs}
\usepackage{amsfonts}
\usepackage{a4wide}
\allowdisplaybreaks \numberwithin{equation}{section}
\usepackage{color}

\newtheorem{theorem}{Theorem}[section]
\newtheorem{proposition}[theorem]{Proposition}

\newtheorem{lemma}[theorem]{Lemma}

\newtheorem{thm}{Theorem}[section]

\newtheorem{lem}[thm]{Lemma}

\newtheorem{prop}[thm]{Proposition}

\theoremstyle{definition}

\newtheorem{remark}[theorem]{Remark}

\newcommand{\R}{\mathbb{R}}
\newcommand{\ds}{\displaystyle}

\begin{document}

\title[bubble solutions for a Schr\"{o}dinger equation with critical growth]
{ New type of positive bubble solutions for a critical Schr\"{o}dinger equation}

 \author{Qihan  He,\,\,\,Chunhua Wang$^{\dagger}$\,\,\,and\,\,\,Qingfang Wang}
\address{College  of  Mathematics and Information Sciences, Guangxi University, Guangxi, 530003, Peopel's Republis of China }
\email{ heqihan277@163.com}
\address{School of Mathematics and Statistics \& Hubei Key Laboratory of Mathematical Sciences, Central China Normal University, Wuhan, 430079, P. R. China }
\email{  chunhuawang@ccnu.edu.cn}
\address{School of Mathematics and Computer Science, Wuhan Polytechnic University, Wuhan, 430023, P. R. China }
\email{ hbwangqingfang@163.com}


\thanks{$^\dagger$ Corresponding author: Chunhua Wang}
 \begin{abstract}
In this paper, we investigate the following critical elliptic equation
$$
-\Delta u+V(y)u=u^{\frac{N+2}{N-2}},\,\,u>0,\,\,\text{in}\,\R^{N},\,\,u\in
H^{1}(\R^{N}),
$$
where
$V(y)$ is a bounded non-negative function in
$\R^{N}.$ Assuming that $V(y)=V(|\hat{y}|,y^{*}),y=(\hat{y},y^{*})\in \R^{4}\times \R^{N-4}$
  and gluing together bubbles with different concentration rates, we
obtain new solutions provided that $N\geq 7,$ whose concentrating points are close to
 the point $(r_{0},y^{*}_{0})$ which is a stable critical point
 of the function $r^{2}V(r,y^{*})$ satisfying $r_{0}>0$ and $V(r_{0},y^{*}_{0})>0.$
In order to construct such new bubble solutions for the above problem,
 we first prove a non-degenerate result for the positive multi-bubbling solutions constructed in
\cite{PWY-18-JFA} by some local Pohozaev identities, which is of great interest independently.
Moreover, we give an example which satisfies the assumptions we impose.



 \end{abstract}

\maketitle

\section{ Introduction and the main result }\label{s1}
Standing waves for the following nonlinear Schr\"{o}dinger equation in $\R^N$,
\begin{equation}\label{e1}
i\frac{\partial\psi}{\partial t}=\Delta\psi-\tilde{V}(y)\psi+|\psi|^{p-1}\psi,
\end{equation}
are solutions of the form $\psi(t,y)=e^{i\lambda t}u(y),$
where $i$ denotes the imaginary
part and $\lambda\in \R,~p>1.$
 Assuming that $u(y)$ is positive and vanishes at infinity, we see
 that $\psi$ satisfies \eqref{e1} if and only if $u$ satisfies the following nonlinear elliptic problem
\begin{equation}\label{e2}
  -\Delta u+V(y)u=u^p,\,\,u>0,\,\,\,\,\,\lim\limits_{|y|\rightarrow\infty}u(y)=0,
\end{equation}
where $V(y)=\tilde{V}(y)-\lambda$. Hereafter, we assume that $V(y)$ is bounded and $V(y)\geq 0.$
When $1<p<\frac{N+2}{N-2}$
in \eqref{e2} i.e. the subcritical exponent case, in \cite{wy}
Wei and Yan showed the equation has infinitely many non-radial positive solutions when $V(y)$ is a radially positive function.
 There are various existence results for the subcritical case, such as \cite{CNY1,CNY2,dp}.

In this paper, we will investigate the critical case i.e. $p=\frac{N+2}{N-2}:$
\begin{equation}\label{1.3}
-\Delta u+V(y)u=u^{\frac{N+2}{N-2}},\,\,u>0,\,\,u\in
H^{1}(\R^{N}),
\end{equation}
where  $V(y)\ge 0$  and  $V\not\equiv 0$.
It corresponds to the following well-known Brezis-Nirenberg problem in $S^N$
\begin{equation}\label{1.3'}
-\Delta_{S^N} u=u^{\frac{N+2}{N-2}}+\mu u,\,\,u>0,\,\,\hbox{on}\,\,S^N.
\end{equation}
Indeed, after using the stereographic projection, problem~\eqref{1.3'} can be reduced to \eqref{1.3} with
$$
V(y)=\frac{-4\mu-N(N-2)}{(1+|y|^2)^2},
$$
and $V(y)>0$ if $\mu<-\frac{N(N-2)}{4}$.
Problem \eqref{1.3'} has been studied extensively.
In \cite{BL},
Brezis and Li proved if $\mu>-\frac{N(N-2)}{4},$ then the only solution to \eqref{1.3'}
is the constant $u=(-\mu)^{\frac{N-2}{4}}.$
When $\mu=-\frac{N(N-2)}{4},$
in \cite{D1} Druet (see also Druet and Hebey \cite{DH1,DH2})
proved that the set of positive solutions to \eqref{1.3'}
is compact provided that the energy is bounded.
Furthermore, in \cite{bw,BP} there have been proved that there are more and more non-radial
solutions as $\mu\rightarrow -\infty.$
In \cite{cwy}, Chen, Wei and Yan proved that when
$\mu<-\frac{N(N-2)}{4}$
and $N\geq 5,$ there are infinitely many non-radial solutions to \eqref{1.3'}
whose energy can be made arbitrarily large. This implies that the boundedness
of energy in \cite{D1,DH1} is necessary. When $\mu=-\frac{N(N-2)}{2}$
  and $u^{\frac{N+2}{N-2}}$ is taken place
  of $K(y)u^{\frac{N+2}{N-2}}$ in \eqref{1.3'}
  where $K(y)$ being a fixed smooth function,  in \cite{WY1}
  Wei and Yan showed that it has infinitely many non-radial positive solutions.
More results on
the existence, multiplicity and qualitative properties of solutions for non-compact elliptic problems can also be found in \cite{BL,DLY,GS,LWX,Re,Re1,WY4} and the references therein.

It is not difficult to see that
if  $V\ge 0$
and $V\not\equiv 0$, then the mountain pass value for problem \eqref{1.3} is not a critical value of the corresponding functional.
Hence all the arguments based on the concentration compactness arguments \cite{li,li1} can not be used to obtain an existence result of solutions  for
\eqref{1.3}.
To our best knowledge, the first existence result for \eqref{1.3}
is due to
 Benci and Cerami \cite{bc}. They proved that  if  $\|V\|_{L^{\frac{N}{2}}(\mathbb R^N)}$ is suitably  small,  \eqref{1.3}
has a solution whose energy is in the interval $\bigl(\frac1N S^{\frac N2}, \frac2N S^{\frac N2}\bigr)$, where $S$ is the best Sobolev
constant in the embedding $D^{1, 2}(\mathbb R^N) \hookrightarrow L^{\frac{2N}{N-2}}(\mathbb R^N)$.
For the Brezis--Nirenberg problem in $S^N$,  Benci and Cerimi's result only yields an existence result if
$-4\mu-N(N-2)>0$ is suitably small.
After \cite{bc}, there is no other result for \eqref{1.3} for a long time
 until  the work by Chen, Wei and Yan \cite{cwy}. In \cite{cwy}, it is
proved that \eqref{1.3} has infinitely many  nonradial
solutions  if  $N\ge 5$, $V(y)$ is radially symmetric and
$r^{2}V(r)$ has a local maximum point, or a local minimum point
$r_{0}>0$ with $V(r_{0})>0$.
Note that this condition is necessary for the existence of solutions since by the following Pohozaev identity
\begin{equation}\label{1.1'}
\int_{\R^N} \Big(V(|y|)+\frac12 |y|V'(|y|) \Big)u^2=0,
\end{equation}
\eqref{1.3} has no solution if $r^2 V(r)$ is always non-decreasing, or non-increasing.
Recently, in \cite{PWY-18-JFA} Peng, Wang and Yan showed that problem \eqref{1.3} has infinitely many solutions
 by
 introducing some local Pohozaev type identities in the finite-dimensional reduction method, where $V(y)$ satisfies
 the following condition:

 $(V')$ suppose that $V(y)=V(|y'|,y'')=V(r,y''), (y',y'')\in \R^{2}\times \R^{N-2}$
 and $r^{2}V(r,y'')$ has a critical point $(r_{0},y''_{0})$ satisfying $r_{0}>0$ and
 $V(r_{0},y''_{0})>0$ and $deg(\nabla (r^{2}V(r,y'')),(r_{0},y''_{0}))\neq 0.$

It is well known that the functions
$$
U_{x,\lambda}(y)=[N(N-2)]^{\frac{N-2}{4}}\Bigl(\frac{\lambda}{1+\lambda^{2}|y-x|^{2}}\Bigr)^{\frac{N-2}{2}},\,\,\,\lambda>0,\,\,
x\in\R^{N},
$$ are the only solutions to the problem
\begin{equation}\label{1.7}
-\Delta u=u^{\frac{N+2}{N-2}},\,\,\,u>0\,\,\,\text{in}\,\,\R^{N}.
\end{equation}

 Define
\[
\begin{split}
 H_{s}=\Bigl\{u:u\in H^{1,2}(\R^{N}),\;\; & u(y_1, -y_2,y'')=
u(y_1, y_2, y''),
 \\
 &
u(r\cos\theta,r\sin\theta,y'')=u\Bigl(r\cos\Bigl(\theta+\frac{2\pi j
}{m}\Bigl),r\sin\Bigl(\theta+\frac{2\pi j
}{m}\Bigl),y''\Bigl)\Bigl\}.
\end{split}
\]
Let
$$
x_{j}=\Bigl(\bar r \cos\frac{2(j-1)\pi}{m},\bar r
\sin\frac{2(j-1)\pi}{m},\bar{y}''\Bigr),\,\,\,j=1,2,\cdots,m,
$$
where $\bar{y}''$ is a vector in $\R^{N-2}.$

 Let
$\delta>0$ be a small constant, such that  $r^2V(r, y'')>0$ if  $|(r,y'')- (r_0, y_0'')|\le
10\delta$.   Let $\zeta(y)= \zeta(|y'|, y'')$ be a
smooth function satisfying $\zeta=1$ if $|(r,y'')- (r_0, y_0'')|\le
\delta$, $\zeta=0$ if $|(r,y'')- (r_0, y_0'')|\ge 2\delta$,  and $0\le
\zeta\le 1$. Denote
\begin{eqnarray*}
Z_{x_{j},\lambda}(y)=\zeta  U_{x_j,\lambda},\,\,\,Z^*_{\bar r,\bar{y}'',\lambda}=\sum_{j=1}^m  U_{x_j,\lambda},\;\;\;
Z_{\bar r,\bar{y}'',\lambda}(y)=\sum_{j=1}^{m}Z_{x_{j},\lambda}(y).
\end{eqnarray*}
Set
\begin{align*}
x_j=\Big(\bar{r}\cos\frac{2(j-1)\pi}{m},\bar{r}\sin\frac{2(j-1)\pi}{m},\bar{y}''\Big),\,j=1,\cdots,m,\,\bar{y}''\in \R^{N-2}.
\end{align*}

\medskip

We recall that the result obtained in \cite{PWY-18-JFA} is as follows.

{\bf Theorem~A.}   {\it Suppose that  $V\ge 0$ is bounded and  belongs to $C^1$.
If  $V(|y'|,y'') $ satisfies $(V')$ and  $N\ge 5$, then
 there exists a
positive integer $m_{0}>0$, such that for any integer $m\geq m_{0}$,
\eqref{1.3} has a solution $u_{m}$ of the form
\begin{equation}\label{sol}
u_{m}=Z_{\bar r_{m},\bar{y}_{m}'',\lambda_{m},}+\varphi_{m}=\sum_{j=1}^{m}\zeta
U_{x_{j},\lambda_{m}}+\varphi_{m},
\end{equation}
where $\varphi_{m}\in H_{s}$. Moreover,  as
$m\rightarrow +\infty$,
$\lambda_{m}\in[L_{0}m^{\frac{N-2}{N-4}},L_{1}m^{\frac{N-2}{N-4}}]$,
$ (\bar r_m,\bar{y}_{m}'')\to (r_{0},y''_{0}),$ and
$\lambda_{m}^{-\frac{N-2}{2}}\|\varphi_{m}\|_{L^{\infty}}\rightarrow
0$.}

To construct new bubble solutions for problem \eqref{1.3},  we first want to apply some local Pohozaev identities
to prove the multi-bubbling solutions in \textbf{Theorem A} above is non-degenerate.

In order to state our main result, we give some assumptions of the function $V(y):$

 ($V$) suppose that $V(y)=V(|\hat{y}|,y^{*})=V(r,y^{*}), (\hat{y},y^{*})\in \R^{4}\times \R^{N-4}$
 and $r^{2}V(r,y^{*})$ has a critical point $(r_{0},y^{*}_{0})$ satisfying $r_{0}>0$ and
 $V(r_{0},y^{*}_{0})>0$ and $deg(\nabla (r^{2}V(r,y^{*})),(r_{0},y^{*}_{0}))\neq 0.$

 ($\tilde{V}$)
 $$
 det(A_{i,l})_{(N-3)\times (N-3)}\neq 0,\,\,i,l=1,2,\cdots,N-3,
$$
where
\begin{equation*}
A_{i,l}=
  \left\{
    \begin{array}{ll}
      \Big[ \frac{\partial^{2}V}{\partial r^{2}}-\big(\frac{\frac{\partial \Delta V}{\partial y_{1}}}{2\Delta V}
+\frac{\nu_{1}}{\langle \nu,x_{1}\rangle}\big)\big(r\frac{\partial^{2}V}{\partial r^{2}}
+\sum\limits_{j=5}^{N}y_{j}\frac{\partial^{2}V}{\partial r\partial y_{j}}\big)\Big](r_{0},y^{*}_{0}),
 \text{when}\,\,i=l=1;\vspace{0.12cm}\\
\Big[ \frac{\partial^{2}V}{\partial r\partial y_{l+3}}-\big(\frac{\frac{\partial \Delta V}{\partial y_{1}}}{2\Delta V}
+\frac{\nu_{1}}{\langle \nu,x_{1}\rangle}\big)\big(r\frac{\partial^{2}V}{\partial r\partial y_{l+3}}
+\sum\limits_{j=5}^{N}y_{j}\frac{\partial^{2}V}{\partial y_{j}\partial y_{l+3}}\big)\Big](r_{0},y^{*}_{0}),  \\
\text{when}\,\,i=1,l=2,3,...,N-3;\vspace{0.12cm}\\
\cos\frac{2i\pi}{m}\Big[ \frac{\partial^{2}V}{\partial r\partial y_{i+3}}-\big(\frac{\frac{\partial \Delta V}{\partial y_{i+3}}}{2\Delta V}
+\frac{\nu_{i+3}}{\langle \nu,x_{1}\rangle}\big)\big(r\frac{\partial^{2}V}{\partial r^{2}}
+\sum\limits_{j=5}^{N}y_{j}\frac{\partial^{2}V}{\partial r\partial y_{j}}\big)
\Big](r_{0},y^{*}_{0}),
\\
\text{when}\,\,i=2,3,...,N-3,l=1;\vspace{0.12cm}\\
      \Big[\frac{\partial^{2}V}{\partial y_{i+3}\partial y_{l+3}}-
\big(\frac{\frac{\partial \Delta V}{\partial y_{i+3}}}{2\Delta V}
+\frac{\nu_{i+3}}{\langle \nu,x_{1}\rangle}\big)\big(r\frac{\partial^{2}V}{\partial r\partial y_{l+3}}
+\sum\limits_{j=5}^{N}y_{j}\frac{\partial^{2}V}{\partial y_{j}\partial y_{l+3}}\big)
\Big](r_{0},y^{*}_{0}),
\\
\text{when}\,\,i,l=2,3,...,N-3,
    \end{array}
  \right.
\end{equation*}
$\nu_{i}$ and $\nu$ are the $i$-th unit outward normal and unit outward normal respectively on $\Omega_{1}$( defined in \eqref{eq-omiga}).

Assume that
$\delta>0$ is a small constant such that  $r^2V(r, y^{*})>0$ if  $|(r,y^{*})- (r_{0},y_{0}^{*})|\le
10\delta$.
We also define
a cut-off function $\hat{\zeta}(y)= \hat{\zeta}(|\hat{y}|, y^{*})$ be a
smooth function satisfying $\hat{\zeta}=1$ if
$|(r,y^{*})- (r_{0},y_{0}^{*})|\le\delta$, $\hat{\zeta}=0$ if $|(r,y^{*})- (r_{0},y_{0}^{*})|\ge 2\delta$,  and $0\le\hat{\zeta}\le 1$.

 \begin{remark}\label{rem-t-8-6}
From the proof of \textbf{Theorem A} in \cite{PWY-18-JFA}, if we substitute the assumption
$(V)$ for the assumption $(V'),$ then only by making some minor modifications we can also prove that
the result of \textbf{Theorem A}
is still true. For simplicity of notations, we still denote the solution as $u_{m},$
and $u_{m}=\sum\limits_{j=1}^m\hat{\zeta}(y)U_{\hat{x}_j,\lambda}+\varphi_{m},$
where
\begin{align*}
\hat{x}_j=\Big(\bar{r}\cos\frac{2(j-1)\pi}{m},\bar{r}\sin\frac{2(j-1)\pi}{m},0,0,\tilde{y}^{*}\Big),\,j=1,\cdots,m.
\end{align*}
\end{remark}

Then, like Remark \ref{rem-t-8-6} we can also find a solution with $n$-bubbles, whose centers lie near
the surface $(r_{0},y_{0}^{*})$ satisfying $|\hat{y}|=|(y_{1},y_{2},y_{3},y_{4})|=r_{0}.$
The question we want to discuss in this paper is whether these two solutions can be glued together
to generate a new type of solutions. In other words, we are concerned with looking for a new solution to \eqref{1.3}, whose shape is, at main order
\begin{align}\label{eqs1.6}
u\approx \sum\limits_{j=1}^m\hat{\zeta}(y)U_{\hat{x}_j,\lambda}+\sum\limits_{j=1}^n\hat{\zeta}(y)U_{p_j,\mu}
:=\sum\limits_{j=1}^mZ_{\hat{x}_j,\lambda}+\sum\limits_{j=1}^nZ_{p_j,\mu},
\end{align}
for $m$ and $n$ big integers, where we take

\begin{align*}
p_j=\Big(0,0,t\cos\frac{2(j-1)\pi}{n},t\sin\frac{2(j-1)\pi}{n},\tilde{y}^*\Big),\,\,j=1,\cdots,n,\,\tilde{y}^*\in \R^{N-4}.
\end{align*}
Here $\bar{r}$ and $t$ are close to $r_0$ and
$\tilde{y}^{*}\rightarrow y^{*}_{0}=(y_{0,5},y_{0,6},...,y_{0,N}).$

The energy functional corresponding to equation \eqref{1.3} is
\[
I(u)=\frac12 \int_{\R^N} \bigl(|\nabla u|^2 +Vu^2\bigr) -\frac1{2^*}\int_{\R^N}|u|^{2^*},\,\,\,\,u\in H^{1}(\R^{N}).
\]
Therefore, generally speaking, a function of the form \eqref{eqs1.6} is an approximate solution to
\eqref{1.3} provided that $\bar{r}, t,\tilde{y}^{*}$ and the parameters $\mu$ and $\lambda$ are such that
\begin{align*}
I'\Big(\sum\limits_{j=1}^mZ_{\hat{x}_j,\lambda}+\sum\limits_{j=1}^nZ_{p_j,\mu}\Big)\sim 0.
\end{align*}
Letting that $\lambda,\mu\rightarrow\infty,$ $\bar{r},t\rightarrow r_{0}$
and
$\tilde{y}^{*}\rightarrow y^{*}_{0},$ we can easily obtain that
\begin{align}\label{es-I}
&I\Big(\sum\limits_{j=1}^mZ_{\hat{x}_j,\lambda}+\sum\limits_{j=1}^nZ_{p_j,\mu}\Big)\cr
&=
(m+n)A+ m\Bigl( \frac{B_{1}V(\bar r,\tilde{y}^{*})}{\lambda^{2}}
-\sum_{j=2}^{m}\frac{B_{2}}{\lambda^{N-2}|\hat{x}_{1}-\hat{x}_{j}|^{N-2}}+    O\bigl( \frac{1}{\lambda^{2+\epsilon}}\bigr)  \Bigr)\cr
&\quad+n \Bigl(\frac{C_{1}V(t,\tilde{y}^{*})}{\mu^{2}}
-\sum_{j=2}^{n}\frac{C_{2}}{\mu^{N-2}|p_{1}-p_{j}|^{N-2}}+    O\bigl( \frac{1}{\mu^{2+\epsilon}}\bigr)  \Bigr),
\end{align}
where $A=\big(\ds\frac{1}{2}-\frac{1}{2^{*}}\big)\int_{\R^{N}}|\nabla U_{0,1}|^{2}$
and $B_1$, $B_2$, $C_1$, $C_2$ are some positive constants and $\epsilon>0$ is a small constant.
Note that if $n\gg m,$
then the two terms in \eqref{es-I} are of different
orders, which causes it not easy to find
a critical point of $I.$ Hence it is very difficult
to apply a reduction argument to construct solutions of the form
\eqref{eqs1.6}.

In this paper, we use a new method which was first introduced
by Guo, Musso, Peng and Yan recently in \cite{GMPY-20-JFA} where
they studied the prescribed scalar curvature equation with a radial potential function.
Recall that we intend to glue $n$-bubbles, whose centers
lie on the surface $(r_{0},y^{*}_{0})$
to the $m$-bubbling solution $u_m$ described in Remark \ref{rem-t-8-6}.
The linear operator for such a problem is
\begin{align*}
Q_n\eta=-\Delta\eta+V(y)\eta-(2^*-1)\Big(u_m+\sum\limits_{j=1}^nZ_{p_j,\mu}\Big)^{2^*-2}\eta.
\end{align*}
Away from the points $p_j$, the operator $Q_n$ can be approximated by the linearized operator around $u_m$, defined by
\begin{align}\label{op-l}
L_m\eta=-\Delta\eta+V(y)\eta-(2^*-1)u_m^{2^*-2}\eta.
\end{align}

 The approach we use here is to construct the solution with $m$-bubbles
 whose center is close to $(r_{0},y^{*}_{0})$
 and $n$-bubbles whose center lies near $(r_{0},y^{*}_{0})$ as a perturbation of the solution with the
$m$-bubbles whose center lie near $(r_{0},y^{*}_{0}).$

The main result of this paper is the following:

\begin{theorem}\label{thm1.3}
Assume that  $V\ge 0$ is bounded and  belongs to $C^3(B_{\varrho}((r_{0},y^{*}_{0}))),$ where $\varrho>0$ is small.
Suppose $V(y)$ satisfies the assumptions $(V),(\tilde{V})$ and $N\geq 7.$ Let $u_m$ be a solution
in \textbf{Remark  \ref{rem-t-8-6}} and $m>0$
is a large even number. Then there is an integer $n_0>0$, depending on $m$, such that for any even number $n\geq n_0,$
 \eqref{1.3} has a solution whose main order is of the form
 \eqref{eqs1.6} for some $t_n\rightarrow r_0, \tilde{y}^{*}\rightarrow y^{*}_{0}$ and $\mu_n\sim n^{\frac{N-2}{N-4}}$.
\end{theorem}

\begin{remark}
Like \cite{PWY-18-JFA}, in section \ref{s3} to deal with the slow decay
of the function $U_{p_{j},\mu}(y)$ when the dimension $N$ is not big,
we introduce the cut-off function $\hat{\zeta}(y).$
\end{remark}

\begin{remark}
We want to point out that if we assume that the small constants $\delta$ in the definition of the cut-off function
$\hat{\zeta}(y)$ and $\vartheta$ in \eqref{thelta}
which are less or equal to $c\mu^{-\frac{1}{2}},$ the result in Theorem \ref{thm1.3} can hold for $N=6.$
It is just technical. In this case, we have the following relation
$$
\frac{1}{\mu^{\frac{1}{2}}}\leq \frac{C}{1+\mu|y-p_{j}|},
$$
which can help us deal with some estimates such as \eqref{w2}.
\end{remark}

\begin{remark}
In Theorem \ref{thm1.3}, in order to obtain the solution $u(y)$ satisfying that it is even about $y_{h},h=1,2,3,4, $
we assume that $m,n$ are both even integers. Otherwise, only to obtain the existence of
the solution $u(y),$ we do not need this requirement.
\end{remark}

To prove Theorem \ref{thm1.3},
first it is very crucial to understand the spectral properties of the linear operator $L_{m}$ and study its
invertibility in some suitable space. We will mainly do this in section \ref{s2}.
Moreover,
since the concentration
points of the bump solutions include a saddle point of $V(r,y^{*}),$
 we can not estimate directly the derivatives of the reduced functional as usual.
 We will apply some local
Pohozaev identities to locate the concentration points of the bump solutions as \cite{PWY-18-JFA,PWW-19-JDE}.
However, in the process of doing the finite-dimensional reduction we need
to compute more carefully, such as the estimate of $J_{4}$ in Lemma \ref{lem3.2},
where we follow some ideas from \cite{GLPY-18-JDE}. Finally, we would like to
point out that the new solutions we construct here which are different from the solutions obtained
in \cite{PWY-18-JFA}.

Our paper is organized as follows. In section~\ref{s2}, we will
prove a non-degenerate result by some local Pohozaev identities,
which is very crucial in constructing a new type of bubbling solutions by applying the finite-dimensional reduction method.
With the non-degenerate result, we construct new solutions and prove Theorem \ref{thm1.3} in section \ref{s3}. In section \ref{sa},
we give some Pohozaev identities.  In section \ref{sb-add}, we discuss the Green function of $L_{m}.$ And we give some basic estimates in Appendix \ref{s5}.
Finally, we give an example of the potential $V(r,y^{*})$ which satisfies the assumptions
$(V)$ and $(\tilde{V})$ in appendix \ref{sb}.

\section{the non-degeneracy of the solutions}\label{s2}
In this section, we mainly prove the non-degeneracy of
the multi-bubbling solutions obtained by Peng, Wang and Yan in \cite{PWY-18-JFA}.

Define
\begin{equation}\label{2.14}
\|u\|_{*}=\sup_{y\in\R^{N}}\Big(\sum_{j=1}^{m}\frac{1}{(1+\lambda_m|y-x_{m,j}|)^{\frac{N-2}{2}+\tau}}\Big)^{-1}\lambda_m^{-\frac{N-2}{2}}|u(y)|
\end{equation}
and
\begin{equation}\label{2.15}
\|f\|_{**}=\sup_{y\in\R^{N}}\Big(\sum_{j=1}^{m}\frac{1}{(1+\lambda_m|y-x_{m,j}|)^{\frac{N+2}{2}+\tau}}\Big)^{-1}\lambda_m^{-\frac{N+2}{2}}|f(y)|,
\end{equation}
where\,$x_{m,j}=(r_m\cos\frac{2(j-1)\pi}{m},r_m\sin\frac{2(j-1)\pi}{m},\bar{x}_{m}'')$, $\tau=\frac{N-4}{N-2}$.

Let
\begin{equation}\label{eq-omiga}
\Omega_j=\Bigl\{y=(y',y'')\in \R^2\times\R^{N-2}: \langle\frac{y{'}}{|y{'}|},\frac{x_{m,j}'}{|x_{m,j}'|}\rangle\geq \cos\frac{\pi}{m}\Bigr\}.
\end{equation}

Define the linear operator
\begin{align}\label{eq-12-1}
L_m\eta=-\Delta\eta+V\eta-(2^*-1)u_m^{2^*-2}\eta.
\end{align}

\begin{lemma}\label{lem2.2}
There is a positive constant $C$ such that
\begin{align*}
|u_m(y)|\leq C\sum\limits_{j=1}^m\frac{\lambda_m^{\frac{N-2}{2}}}{1+(\lambda_m|y-x_{m,j}|)^{N-2}}\,\,
\text{for~all}\,\,\,y\in \R^{N}.
\end{align*}
\end{lemma}

\begin{proof}
Since the proof is just the same as  Lemma 2.2 of \cite{GMPY-20-JFA}, here we omit it.
\end{proof}

The main result of this section is the following.
\begin{prop}\label{thm1.2}
Suppose $N\geq 5$. Assume that $V(y)$ satisfies $(V')$ and ($\tilde{V}'$).
Let $\eta\in H_s$ be a solution of $L_m\eta=0.$ Then for large $m$ there holds $\eta=0.$
\end{prop}

Now we will prove Proposition \ref{thm1.2} by an indirect method. Assume that there are $m_k\rightarrow+\infty$, satisfying $\|\eta_k\|_*=1$  and
\begin{equation}\label{2.17}
L_{m_k}\eta_k=0.
\end{equation}

Denote
\begin{equation}\label{2.18}
\widetilde{\eta}_k(y)=\lambda_{m_k}^{-\frac{N-2}{2}}\eta_k(\lambda_{m_k}^{-1}y+x_{m_k,1}).
\end{equation}

\begin{lemma}\label{lem2.3}
It holds
\begin{equation}\label{2.19}
\widetilde{\eta}_k\rightarrow b_0\psi_0+\sum\limits_{i=1,i\neq2}^Nb_i\psi_i,
\end{equation}
uniformly in $C^1(B_R(0))$ for any $R>0$, where $b_0$ and $b_i(i=1,3,\cdots,N)$ are some constants,
\begin{align*}
\psi_0=\frac{\partial U_{0,\lambda}}{\partial\lambda}\Big|_{\lambda=1},\,\,\,\,\,\,\,\psi_i=\frac{\partial U_{0,1}}{\partial y_i},\,\,i=1,3,\cdots,N.
\end{align*}
\end{lemma}

\begin{proof}
Observing that $|\widetilde{\eta}_k|\leq C$, we may suppose that $\widetilde{\eta}_k\rightarrow\eta$ in $C_{loc}(\R^N)$. Then $\eta$ satisfies
$$
-\Delta\eta=(2^*-1)U^{2^*-2}\eta,\,\,\,x\in\R^N,
$$
which implies
$$
\eta=\sum\limits_{i=0}^Nb_i\psi_i.
$$
Since $\eta_k$ is even in $y_2$, there holds $b_2=0$.
\end{proof}

We decompose
\begin{align*}
\eta_k(y)=&b_{0,m}\lambda_{m_k}\sum\limits_{j=1}^{m_k}\frac{\partial Z_{x_{{m_k},j},\lambda_{m_k}}}{\partial\lambda_{m_k}}+b_{1,m}\lambda_{m_k}^{-1}\sum\limits_{j=1}^{m_k}\frac{\partial Z_{x_{{m_k},j},\lambda_{m_k}}}{\partial \overline{r}}\cr
&+\sum\limits_{i=3}^{N}b_{i,m}\lambda_{m_k}^{-1}\frac{\partial Z_{x_{{m_k},j},\lambda_{m_k}}}{\partial \bar{y}_i}+\eta_k^*,
\end{align*}
where $\eta_k^*$ satisfies
\begin{align*}
\int_{\R^N}Z_{x_{{m_k},j},\lambda_{m_k}}^{2^*-2}\frac{\partial Z_{x_{{m_k},j},\lambda_{m_k}}}{\partial\lambda_{m_k}}\eta_k^*&
=\int_{\R^N}Z_{x_{{m_k},j},\lambda_{m_k}}^{2^*-2}\frac{\partial Z_{x_{{m_k},j},\lambda_{m_k}}}{\partial \overline{r}}\eta_k^*\cr
&=\int_{\R^N}Z_{x_{{m_k},j},\lambda_{m_k}}^{2^*-2}\frac{\partial Z_{x_{{m_k},j},\lambda_{m_k}}}{\partial \bar{y}_i}\eta_k^*=0,\,\,(i=3,\cdots,N).
\end{align*}
It follows from Lemma \ref{lem2.3} that  $b_{0,m}$,$b_{1,m}$ and $b_{i,m}(i=3,\cdots,N)$ are bounded.

\begin{lemma}
There holds
$$
\|\eta_k^*\|_{*}\leq C \lambda_{m_k}^{-1-\epsilon},
$$
where $\epsilon>0$ is a small constant.
\end{lemma}

\begin{proof}
One can see easily that
\begin{align*}
&L_{m_k}\eta_k^*\\
&=-\Delta\eta_k^*+V \eta_k^*-(2^*-1)u_{m_k}^{2^*-2}\eta_k^*\cr
&=-V \Big(b_{0,m}\lambda_{m_k}\sum\limits_{j=1}^{m_k}\frac{\partial Z_{x_{{m_k},j},\lambda_{m_k}}}{\partial\lambda_{m_k}}+\frac{b_{1,m}}{\lambda_{m_k}}\frac{\partial Z_{x_{{m_k},j},\lambda_{m_k}}}{\partial \overline{r}}
+\sum\limits_{i=3}^N\frac{b_{i,m}}{\lambda_{m_k}}\sum\limits_{j=1}^{m_k}\frac{\partial Z_{x_{{m_k},j},\lambda_{m_k}}}{\partial \bar{y}_i}\Big)\cr
&\quad+(2^*-1)\sum\limits_{j=1}^{m_k}(u_{m_k}^{2^*-2}-U_{x_{{m_k},j},\lambda_{m_k}}^{2^*-2})
\Big(b_{0,m}\lambda_{m_k}\frac{\partial Z_{x_{{m_k},j},\lambda_{m_k}}}{\partial\lambda_{m_k}}\cr
&
\quad\quad\quad
+\frac{b_{1,m}}{\lambda_{m_k}}\frac{\partial Z_{x_{{m_k},j},\lambda_{m_k}}}{\partial r}+\sum\limits_{i=3}^N\frac{b_{i,m}}{\lambda_{m_k}}\frac{\partial Z_{x_{{m_k},j},\lambda_{m_k}}}{\partial \bar{y}_i}\Big)\cr
&\quad
+2\nabla \xi\Big(b_{0,m}\lambda_{m_k}\sum\limits_{j=1}^{m_k}\nabla\frac{\partial U_{x_{{m_k},j},\lambda_{m_k}}}{\partial \lambda_{m_k}}
+\frac{b_{1,m}}{\lambda_{m_k}}\sum\limits_{j=1}^{m_k}\nabla\frac{\partial U_{x_{{m_k},j},\lambda_{m_k}}}{\partial r}+\sum\limits_{i=3}^N\frac{b_{i,m}}{\lambda_{m_k}}\sum\limits_{j=1}^{m_k}\nabla\frac{\partial U_{x_{{m_k},j},\lambda_{m_k}}}{\partial \bar{y}_i}\Big)\cr
&\quad
+\Delta\xi \Big(b_{0,m}\lambda_{m_k}\sum\limits_{j=1}^{m_k}\frac{\partial U_{x_{{m_k},j},\lambda_{m_k}}}{\partial \lambda_{m_k}}
+\frac{b_{1,m}}{\lambda_{m_k}}\sum\limits_{j=1}^{m_k}\frac{\partial U_{x_{{m_k},j},\lambda_{m_k}}}{\partial r}+\sum\limits_{i=3}^N\frac{b_{i,m}}{\lambda_{m_k}}\sum\limits_{j=1}^{m_k}\frac{\partial U_{x_{{m_k},j},\lambda_{m_k}}}{\partial \bar{y}_i}\Big)\\
:&=L_{1}+L_{2}+L_{3}+L_{4}.
\end{align*}
Similar to the proof of $J_{1}$  in Lemma 2.4 in \cite{PWY-18-JFA}, we can prove
\begin{align}\label{eqs2.16}
\bigl\|L_{1}\bigr\|_{**}\leq C\lambda_{m_k}^{-1-\epsilon}.
\end{align}
By the same argument as that of \cite{GMPY-20-JFA} in Lemma 2.4, we can estimate
\begin{align}\label{eqs2.17}
&\bigl\|L_{2}\bigr\|_{**}\leq C\lambda_{m_k}^{-1-\epsilon}.
\end{align}
Similar to  $J_{2}$ of Lemma 2.4 in \cite{PWY-18-JFA}, we can check
\begin{align*}
|L_3|
&\leq C\Big(\frac{1}{\lambda_{m_k}}\Big)^{1+\epsilon}
\sum\limits_{j=1}^{m_k}\frac{\lambda_{m_k}^{\frac{N+2}{2}}}{(1+\lambda_{m_k}|y-x_{{m_k},j}|)^{\frac{N+2}{2}+\tau}}.
\end{align*}
Hence, we obtain
\begin{align}\label{eqs2.18}
\|L_3\|_{**}\leq C\Big(\frac{1}{\lambda_{m_k}}\Big)^{1+\epsilon}.
\end{align}
Also, similar to $J_3$ of Lemma 2.4 in \cite{PWY-18-JFA}, we can prove
\begin{align*}
|L_4|
\leq &C\Big(\frac{1}{\lambda_{m_k}}\Big)^{1+\epsilon}
\sum\limits_{j=1}^{m_k}\frac{\lambda_{m_k}^{\frac{N+2}{2}}}{(1+\lambda_{m_k}|y-x_{{m_k},j}|)^{\frac{N+2}{2}+\tau}},
\end{align*}
which yields that
\begin{align}\label{eqs2.19}
\|L_4\|_{**}\leq C\Big(\frac{1}{\lambda_{m_k}}\Big)^{1+\epsilon}.
\end{align}
It follows from \eqref{eqs2.16} to \eqref{eqs2.19} that
\begin{align}\label{w2.19}
\|L_{m_k}\eta_k^*\|_{**}\leq C\lambda_{m_k}^{-1-\epsilon}.
\end{align}
Furthermore, from
\begin{align*}
\int_{\R^N}Z_{x_{{m_k},j},\lambda_{m_k}}^{2^*-2}\frac{\partial Z_{x_{{m_k},j},\lambda_{m_k}}}{\partial\lambda_{m_k}}\eta_k^*
&=\int_{\R^N}Z_{x_{{m_k},j},\lambda_{m_k}}^{2^*-2}\frac{\partial Z_{x_{{m_k},j},\lambda_{m_k}}}{\partial \overline{r}}\eta_k^*\cr
&=\int_{\R^N}Z_{x_{{m_k},j},\mu_{m_k}}^{2^*-2}\frac{\partial Z_{x_{{m_k},j},\lambda_{m_k}}}{\partial \bar{y}_i}\eta_k^*
=0,\,\,(i=3,\cdots,N)
\end{align*}
and Lemma \ref{lem2.2}, we can prove that there exists $\rho>0$ such that
\begin{align}\label{ww2.19}
\|L_{m_k}\eta_k^*\|_{**}\geq \rho\|\eta_k^*\|_{*}.
\end{align}
Combining \eqref{w2.19} and \eqref{ww2.19}, the result is true.
\end{proof}

Now we give another assumption of $V(y):$

 ($\tilde{V}'$)
$$
 det(A_{i,l})_{(N-1)\times (N-1)}\neq 0,\,\,i,l=1,2,...,N-1,
$$
where
\begin{equation}\label{eq-matrix}
A_{i,l}=
  \left\{
    \begin{array}{ll}
      \Big[ \frac{\partial^{2}V}{\partial r^{2}}-\big(\frac{\frac{\partial \Delta V}{\partial y_{1}}}{2\Delta V}
+\frac{\nu_{1}}{\langle \nu,x_{1}\rangle}\big)\big(r\frac{\partial^{2}V}{\partial r^{2}}
+\sum\limits_{j=3}^{N}y_{j}\frac{\partial^{2}V}{\partial r\partial y_{j}}\big)\Big](r_{0},y''_{0}),
 \text{when}\,\,i=l=1;
\vspace{0.12cm}\\
\Big[ \frac{\partial^{2}V}{\partial r\partial y_{l+1}}-\big(\frac{\frac{\partial \Delta V}{\partial y_{1}}}{2\Delta V}
+\frac{\nu_{1}}{\langle \nu,x_{1}\rangle}\big)\big(r\frac{\partial^{2}V}{\partial r\partial y_{l+1}}
+\sum\limits_{j=3}^{N}y_{j}\frac{\partial^{2}V}{\partial y_{j}\partial y_{l+1}}\big)\Big](r_{0},y''_{0}),
\\
\text{when}\,\,i=1,
l=2,3,...,N-1;
\vspace{0.12cm}\\
\cos\frac{2i\pi}{m}\Big[ \frac{\partial^{2}V}{\partial r\partial y_{i+1}}-\big(\frac{\frac{\partial \Delta V}{\partial y_{i+1}}}{2\Delta V}
+\frac{\nu_{i+1}}{\langle \nu,x_{1}\rangle}\big)\big(r\frac{\partial^{2}V}{\partial r^{2}}
+\sum\limits_{j=3}^{N}y_{j}\frac{\partial^{2}V}{\partial r\partial y_{j}}\big)
\Big](r_{0},y''_{0}), \\
 \text{when}\,\,i=2,3,...,N-1,l=1;
\vspace{0.12cm}\\
     \Big[\frac{\partial^{2}V}{\partial y_{i+1}\partial y_{l+1}}-
\big(\frac{\frac{\partial \Delta V}{\partial y_{i+1}}}{2\Delta V}
+\frac{\nu_{i+1}}{\langle \nu,x_{1}\rangle}\big)\big(r\frac{\partial^{2}V}{\partial r\partial y_{l+1}}
+\sum\limits_{j=3}^{N}y_{j}\frac{\partial^{2}V}{\partial y_{j}\partial y_{l+1}}\big)
\Big](r_{0},y''_{0}),
\\
\text{when}\,\,i,l=2,3,...,N-1,
    \end{array}
  \right.
\end{equation}
$\nu_{i}$ and $\nu$ are the $i$-th unit outward normal and unit outward normal respectively on $\Omega_{1}$ defined in \eqref{eq-omiga}.

\begin{lemma}
If  $(\tilde{V'})$ holds, then
$$
\tilde{\eta}_k\rightarrow0
$$
uniformly in $C^1(B_R(0))$ for any $R>0$.
\end{lemma}
\begin{proof}
Step1. Recall that
\begin{align*}
\Omega_j=\Bigl\{y=(y',y'')\in \R^2\times\R^{N-2}: \langle\frac{y{'}}{|y{'}|},\frac{x_{m,j}'}{|x_{m,j}'|}\rangle\geq \cos\frac{\pi}{m}\Bigr\}.
\end{align*}
In order to prove $b_{i,k}\rightarrow0(i=1,3,\cdots,N)$, we apply the identities in Lemma
\ref{lem2.1} in the domain $\Omega_1,$
\begin{align}\label{2.5.2}
&-\int_{\partial\Omega_1}\frac{\partial u_{m_k}}{\partial\nu}\frac{\partial \eta_k}{\partial y_i}-\int_{\partial\Omega_1}\frac{\partial \eta_k}{\partial\nu}\frac{\partial u_{m_k}}{\partial y_i}+\int_{\partial\Omega_1}\langle\nabla u_{m_k},\nabla\eta_k\rangle\nu_i
+\int_{\partial\Omega_1}Vu_{m_k}\eta_k\nu_i\cr
&\,\,\,-\int_{\partial\Omega_1}u_{m_k}^{2^*-1}\eta_k\nu_i
=\int_{\Omega_1}\frac{\partial V}{\partial y_i}u_{m_k}\eta_k,\,\,\,i=1,3,\cdots,N.
\end{align}
By the symmetry, we have $\frac{\partial u_{m_k}}{\partial\nu}=0$ and $\frac{\partial\eta_k}{\partial\nu}=0$ on $\partial\Omega_1.$
Hence
\begin{align}\label{2.5.3}
&\hbox{the left hand side of}~ \eqref{2.5.2}\cr
&=\nu_{i}\Big(\int_{\partial\Omega_1}\langle\nabla u_{m_k},\nabla \eta_k\rangle+\int_{\partial\Omega_1}Vu_{m_k}\eta_k-\int_{\partial\Omega_1}u_{m_k}^{2^*-1}\eta_k\Big).
\end{align}

Combining \eqref{2.5.2} and \eqref{2.5.3}, we obtain
\begin{align}\label{2.5.4}
\nu_{i}\Big(\int_{\partial\Omega_1}\langle\nabla u_{m_k},\nabla \eta_k\rangle+\int_{\partial\Omega_1}Vu_{m_k}\eta_k-\int_{\partial\Omega_1}u_{m_k}^{2^*-1}\eta_k\Big)
=\int_{\Omega_1}\frac{\partial V}{\partial y_i}u_{m_k}\eta_k.
\end{align}
To estimate the left hand side in \eqref{2.5.4}, we use \eqref{eqs2.4} in $\Omega_1.$ Applying the symmetry, we have
\begin{align}\label{2.5.5}
&\int_{\Omega_1}u_{m_k}\eta_k\langle\nabla V,y-x_{m_k,1}\rangle+2\int_{\Omega_1}V\eta_ku_{m_k}
=-\int_{\partial\Omega_1}u_{m_k}^{2^*-1}\eta_k\langle\nu,y-x_{m_k,1}\rangle\cr
&+\int_{\partial\Omega_1}\langle\nabla u_{m_k},\nabla\eta_k\rangle\langle\nu,y-x_{m_k,1}\rangle+\int_{\partial\Omega_1}V u_{m_k}\eta_k\langle\nu,y-x_{m_k,1}\rangle.
\end{align}
On $\partial\Omega_1,$ it holds $\langle\nu,y\rangle=0.$
Then, \eqref{2.5.5} becomes
\begin{align}\label{eqs2.24}
&\int_{\Omega_1}u_{m_k}\eta_k\langle\nabla V,y-x_{m_k,1}\rangle+2\int_{\Omega_1}V \eta_ku_{m_k}\cr
=&-\langle\nu,x_{m_k,1}\rangle\Big(-\int_{\partial\Omega_1}u_{m_k}^{2^*-1}\eta_k+\int_{\partial\Omega_1}\langle\nabla u_{m_k},\nabla\eta_k\rangle+\int_{\partial\Omega_1}V u_{m_k}\eta_k\Big).
\end{align}
It follows from \eqref{2.5.4} and \eqref{eqs2.24} that
\begin{align}\label{eqs2.26}
\int_{\Omega_1}\frac{\partial V}{\partial y_i}u_{m_k}\eta_k=\frac{-\nu_i}{\langle \nu,x_{m_k,1}\rangle}\Big(\int_{\Omega_1}u_{m_k}\eta_k\langle\nabla V,y-x_{m_k,1}\rangle+2\int_{\Omega_1}V \eta_ku_{m_k}\Big).
\end{align}
Since $\nabla V(x_{m_k,1})=O(|x_{m_k,1}-y_0|)$ and
\begin{align}\label{eqs2.27}
\int_{\Omega_1}u_{m_k}\eta_k&=\int_{(\Omega_1)_{x_{m_k,1},\lambda_{m_k}}}
\big(\lambda_{m_k}^{-\frac{N-2}{2}}u_{m_k}(\lambda_{m_k}^{-1}y+x_{m_k,1})\big)\tilde{\eta}_k\cr
&=\frac{1}{\lambda_{m_k}^2}\int_{\R^N}U\big[b_{0,m}\psi_0+b_{1,m}\psi_1+\sum\limits_{i=3}^Nb_{i,m}\psi_i
+\lambda_{m_k}^{-\frac{N-2}{2}}\eta_k^*(\lambda_{m_k}^{-1}y+x_{m_k,1})\big]+O\big(\lambda_{m_k}^{-3-\epsilon}\big)\cr
&=O\big(\lambda_{m_k}^{-3-\epsilon}\big),
\end{align}
where $(\Omega_1)_{x_{k_m},1}=\{y, \lambda_{k_m}^{-1}y+x_{k_m,1}\in \Omega_1\},$ we have
\begin{align}\label{eqs2.28}
 \int_{\Omega_1}\frac{\partial V}{\partial y_i}u_{m_k}\eta_k=&
 \int_{\Omega_1}u_{m_k}\eta_k\Big(\frac{\partial V}{\partial y_i}-\frac{\partial V(x_{m_k,1})}{\partial y_i}\Big)
 +\int_{\Omega_1}\frac{\partial V(x_{m_k,1})}{\partial y_i}u_{m_k}\eta_k\cr
 =&\int_{\Omega_1}u_{m_k}\eta_k\Big[\langle\nabla \frac{\partial V(x_{m_k,1})}{\partial y_i},y-x_{m_k,1}\rangle\cr
 &+\frac{1}{2}\langle\nabla^2 \frac{\partial V(x_{m_k,1})}{\partial y_i}(y-x_{m_k,1}),y-x_{m_k,1}\rangle+O(|y-x_{m_k,1}|^3)\Big]+O(\lambda_{m_k}^{-4-\epsilon})\cr
 =&\frac{1}{\lambda_{m_k}^2}\int_{\R^N}U\Big(b_{0,k}\psi_0+b_{1,k}\psi_1+
 \sum\limits_{l=3}^Nb_{l,k}\psi_l\Big)\Big(\langle\nabla\frac{\partial V(x_{m_k,1})}{\partial y_i},\frac{y}{\lambda_{m_k}}\rangle\cr
 &+\frac{1}{2}\langle\nabla^2\frac{\partial V(x_{m_k,1})}{\partial y_i}\frac{y}{\lambda_{m_k}},\frac{y}{\lambda_{m_k}}\rangle\Big)+O(\lambda_{m_k}^{-4-\epsilon})\cr
 =&\frac{\frac{\partial^{2} V}{\partial y_{1}\partial y_{i}}(x_{m_k,1})}{\lambda_{m_k}^3}b_{1,k}\int_{\R^N}U \psi_1y_1+\sum\limits_{l=3}^Nb_{l,k}\frac{\frac{\partial^{2} V}{\partial y_{l}\partial y_{i}}(x_{m_k,1})}{\lambda_{m_k}^3}\int_{\R^N}U\psi_ly_l\cr
 &+\frac{\frac{\partial\triangle V(x_{m_k,1})}{\partial y_i}b_{0,k}}{2N\lambda_{m_k}^4}\int_{\R^N}U\psi_0|y|^2
 +O\big(\lambda_{m_k}^{-4-\epsilon}\big).
\end{align}
Moreover, we can estimate
\begin{align}\label{eqs2.29}
&\int_{\Omega_1}u_{m_k}\eta_k\langle\nabla V,y-x_{m_k,1}\rangle\cr
&=\int_{\Omega_1}u_{m_k}\eta_k\langle\nabla V(y)-\nabla V(x_{m_k,1}),y-x_{m_k,1}\rangle+\int_{\Omega_1}u_{m_k}\eta_k\langle\nabla V(x_{m_k,1}),y-x_{m_k,1}\rangle\cr
&=\int_{\Omega_1}u_{m_k}\eta_k\langle\nabla V(y)-\nabla V(k_{m_k,1}),y-x_{m_k,1}\rangle+O(\lambda_{m_k}^{-4-\epsilon})\cr
&=\int_{\Omega_1}u_{m_k}\eta_k\langle\nabla^2V(x_{m_k,1})(y-x_{m_k,1}),y-x_{m_k,1}\rangle
+O(\lambda_{m_k}^{-4-\epsilon})\cr
&=\frac{1}{\lambda_{m_k}^2}\int_{\R^N}U\Big(b_{0,k}\psi_0+b_{1,k}\psi_1+\sum\limits_{l=3}^Nb_{l,k}\psi_l\Big)
\big\langle\nabla^2V(x_{m_k,1})\lambda_{m_k}^{-1}y,\lambda_{m_k}^{-1}y\big\rangle
+O(\lambda_{m_k}^{-4-\epsilon})\cr
&=\frac{b_{0,k}\Delta V(x_{m_k,1})}{N\lambda_{m_k}^4}\int_{\R^N}U\psi_0|y|^2+O\big(\lambda_{m_k}^{-4-\epsilon}\big).
\end{align}
Hence, \eqref{eqs2.28} and \eqref{eqs2.29} give
\begin{align}\label{eqs2.30}
&b_{0,k}\frac{1}{\lambda_{m_{k}}}\Big(\frac{\frac{\partial\Delta V}{\partial y_{i}}(x_{m_k,1})}{2N}
+\frac{\nu_{i}}{\langle \nu,x_{m_{k},1}\rangle}\frac{\Delta V(x_{m_k,1})}{N}\Big)\ds\int_{\R^N}U\psi_0|y|^2
\cr
&+b_{1,k}\frac{\partial^{2}V}{\partial y_{1}\partial y_{i}}(x_{m_k,1})\int_{\R^N}U \psi_1y_1
+\sum\limits_{l=3}^Nb_{l,k}\frac{\partial^{2}V}{\partial y_{l}\partial y_{i}}(x_{m_k,1})\int_{\R^N}U\psi_ly_l
=O\big(\lambda_{m_k}^{-1-\epsilon}\big).
\end{align}

Step 2.  Next, we apply \eqref{eqs2.4} to get
\begin{align*}
 \int_{\R^{N}}u_{m_k}\eta_k\langle\nabla V(y),y\rangle=0,
\end{align*}
which implies
\begin{equation}\label{888}
\ds\int_{\Omega_{i}}u_{m_k}\eta_k\langle\nabla V(y),y\rangle=0.
\end{equation}

On the other hand, proceeding as in the proof of \eqref{eqs2.27}, we have
\begin{align*}
&\int_{\Omega_1}u_{m_k}\eta_k\langle\nabla V(x_{m_k,1}),y\rangle\cr
&=\int_{\Omega_1}u_{m_k}\eta_k\langle\nabla V(x_{m_k,1}),y-x_{m_k,1}\rangle+\int_{\Omega_1}u_{m_k}\eta_k\langle\nabla V(x_{m_k,1}),x_{m_k,1}\rangle\cr
&=O\big(\lambda_{m_k}^{-4-\epsilon}\big).
\end{align*}
Therefore, from \eqref{eqs2.30}, we have
\begin{align*}
&\int_{\Omega_1}u_{m_k}\eta_k\langle\nabla V(y),y\rangle
=\int_{\Omega_1}u_{m_k}\eta_k\langle\nabla V(y)-\nabla V(x_{m_k,1}),y\rangle+O\big(\lambda_{m_k}^{-4-\epsilon}\big)\cr
=&\int_{\Omega_1}u_{m_k}\eta_k\langle\nabla^2V(x_{m_k,1})(y-x_{m_k,1}),y\rangle+O\big(\lambda_{m_k}^{-4-\epsilon}\big)\cr
=&\frac{1}{\lambda_{m_k}^2}\int_{\R^N}U \Big(b_{0,k}\psi_0+b_{1,k}\psi_1+\sum\limits_{l=3}^Nb_{l,k}\psi_l\Big)\langle\nabla^2V(x_{m_k,1})\lambda_{m_k}^{-1}y, \lambda_{m_k}^{-1}y+x_{m_k,1}\rangle+O\big(\lambda_{m_k}^{-4-\epsilon}\big)\cr
=&\frac{b_{0,k}\Delta V(x_{m_k,1})}{N\lambda_{m_k}^4}\int_{\R^N}U \psi_0|y|^2
+
\frac{b_{1,k}}{\lambda_{m_{k}}^3}\sum_{j=1}^{N}(x_{m_k,1})_{j}\frac{\partial^{2} V}{\partial y_{j}\partial y_{1}}(x_{m_k,1})
\int_{\R^N}U\psi_1y_{1}
\\
&+\frac{\sum\limits_{l=3}^Nb_{l,k}\ds\sum_{j=1}^{N}(x_{m_k,1})_{j}
\frac{\partial^{2} V}{\partial y_{j}\partial y_{l}}(x_{m_k,1})}{\lambda_{m_k}^3}
\int_{\R^N}U\psi_ly_{l}
+O\big(\lambda_{m_k}^{-4-\epsilon}\big),
\end{align*}
which combining with \eqref{888} implies that
\begin{align}\label{eqs2.36}
&b_{0,k}\frac{\Delta V(x_{m_k,1})}{N\lambda_{m_{k}}}\int_{\R^N}U \psi_0|y|^2
+
b_{1,k}\sum_{j=1}^{N}(x_{m_k,1})_{j}\frac{\partial^{2}V}{\partial y_{j}\partial y_{1}}(x_{m_k,1})\int_{\R^N}U \psi_1y_1
\cr&
+\sum\limits_{l=3}^Nb_{l,k}\sum_{j=1}^{N}(x_{m_k,1})_{j}\frac{\partial^{2}V}{\partial y_{j}\partial y_{l}}(x_{m_k,1})\int_{\R^N}U\psi_ly_l
=O\big(\lambda_{m_k}^{-1-\epsilon}\big).
\end{align}
It follows from \eqref{eqs2.30} and \eqref{eqs2.36} that
\begin{align*}\label{eqs2.37}
&b_{1,k}
\Big[\frac{\partial^{2}V}{\partial y_{1}\partial y_{i}}(x_{m_k,1})-
\big(\frac{\frac{\partial \Delta V}{\partial y_{i}}(x_{m_k,1})}{2\Delta V(x_{m_k,1})}
+\frac{\nu_{i}}{\langle \nu,x_{m_{k},1}\rangle}\big)
\sum_{j=1}^{N}(x_{m_k,1})_{j}\frac{\partial^{2}V}{\partial y_{j}\partial y_{1}}(x_{m_k,1})
\Big]\int_{\R^N}U \psi_1y_1
\cr
&+\sum\limits_{l=3}^Nb_{l,k}
\Big[
\frac{\partial^{2}V}{\partial y_{l}\partial y_{i}}(x_{m_k,1})
-\big(\frac{\frac{\partial \Delta V}{\partial y_{i}}(x_{m_k,1})}{2\Delta V(x_{m_k,1})}
+\frac{\nu_{i}}{\langle \nu,x_{m_{k},1}\rangle}\big)
\sum_{j=1}^{N}(x_{m_k,1})_{j}\frac{\partial^{2}V}{\partial y_{j}\partial y_{l}}(x_{m_k,1})
\Big]
\int_{\R^N}U\psi_ly_l
\cr&
=O\big(\lambda_{m_k}^{-1-\epsilon}\big),\,\,i=1,3,4,...,N,
\end{align*}
which implies that
\begin{equation}\label{eqs-main}
  \left\{
    \begin{array}{ll}
      \Big[ \frac{\partial^{2}V}{\partial r^{2}}-\big(\frac{\frac{\partial \Delta V}{\partial y_{1}}}{2\Delta V}
+\frac{\nu_{1}}{\langle \nu,x_{1}\rangle}\big)\big(r\frac{\partial^{2}V}{\partial r^{2}}
+\sum\limits_{j=3}^{N}y_{j}\frac{\partial^{2}V}{\partial y_{r}\partial y_{j}}\big)\Big](r_{0},y''_{0})b_{1,k}
\\
+\sum\limits_{l=2}^{N-1}\Big[ \frac{\partial^{2}V}{\partial r\partial y_{l+1}}-\big(\frac{\frac{\partial \Delta V}{\partial y_{1}}}{2\Delta V}
+\frac{\nu_{1}}{\langle \nu,x_{1}\rangle}\big)\big(r\frac{\partial^{2}V}{\partial r\partial y_{l+1}}
+\sum\limits_{j=3}^{N}y_{j}\frac{\partial^{2}V}{\partial y_{j}\partial y_{l+1}}\big)\Big](r_{0},y''_{0})b_{l+1,k}=O\big(\lambda_{m_k}^{-1-\epsilon}\big)
\vspace{0.12cm}\\
\cos\frac{4\pi}{m}\Big[ \frac{\partial^{2}V}{\partial r\partial y_{3}}-\big(\frac{\frac{\partial \Delta V}{\partial y_{3}}}{2\Delta V}
+\frac{\nu_{3}}{\langle \nu,x_{1}\rangle}\big)\big(r\frac{\partial^{2}V}{\partial r^{2}}
+\sum\limits_{j=3}^{N}y_{j}\frac{\partial^{2}V}{\partial y_{r}\partial y_{j}}\big)
\Big](r_{0},y''_{0})b_{1,k}
\\
+\sum\limits_{l=2}^{N-1} \Big[\frac{\partial^{2}V}{\partial y_{3}\partial y_{l+1}}-
\big(\frac{\frac{\partial \Delta V}{\partial y_{3}}}{2\Delta V}
+\frac{\nu_{3}}{\langle \nu,x_{1}\rangle}\big)\big(r\frac{\partial^{2}V}{\partial r\partial y_{l+1}}
+\sum\limits_{j=3}^{N}y_{j}\frac{\partial^{2}V}{\partial y_{j}\partial y_{l+1}}\big)
\Big](r_{0},y''_{0})b_{l+1,k}=O\big(\lambda_{m_k}^{-1-\epsilon}\big)
\vspace{0.12cm}\\
................................
\vspace{0.12cm}\\
\cos\frac{2i\pi}{m}\Big[ \frac{\partial^{2}V}{\partial r\partial y_{i+1}}-\big(\frac{\frac{\partial \Delta V}{\partial y_{i+1}}}{2\Delta V}
+\frac{\nu_{i+1}}{\langle \nu,x_{1}\rangle}\big)\big(r\frac{\partial^{2}V}{\partial r^{2}}
+\sum\limits_{j=3}^{N}y_{j}\frac{\partial^{2}V}{\partial r\partial y_{j}}\big)
\Big](r_{0},y''_{0})b_{1,k}
\\
+\sum\limits_{l=2}^{N-1} \Big[\frac{\partial^{2}V}{\partial y_{i+1}\partial y_{l+1}}-
\big(\frac{\frac{\partial \Delta V}{\partial y_{i+1}}}{2\Delta V}
+\frac{\nu_{i+1}}{\langle \nu,x_{1}\rangle}\big)\big(r\frac{\partial^{2}V}{\partial r\partial y_{l+1}}
+\sum\limits_{j=3}^{N}y_{j}\frac{\partial^{2}V}{\partial y_{j}\partial y_{l+1}}\big)
\Big](r_{0},y''_{0})b_{l+1,k}=O\big(\lambda_{m_k}^{-1-\epsilon}\big)
\vspace{0.12cm}\\
.................................
\vspace{0.12cm}\\
    \cos\frac{2(N-1)\pi}{m} \Big[ \frac{\partial^{2}V}{\partial r\partial y_{N}}-\big(\frac{\frac{\partial \Delta V}{\partial y_{N}}}{2\Delta V}
+\frac{\nu_{N}}{\langle \nu,x_{1}\rangle}\big)\big(r\frac{\partial^{2}V}{\partial r^{2}}
+\sum\limits_{j=3}^{N}y_{j}\frac{\partial^{2}V}{\partial r\partial y_{j}}\big)
\Big](r_{0},y''_{0})b_{1,k}
\\
+\sum\limits_{l=2}^{N-1} \Big[\frac{\partial^{2}V}{\partial y_{N}\partial y_{l+1}}-
\big(\frac{\frac{\partial \Delta V}{\partial y_{N}}}{2\Delta V}
+\frac{\nu_{N}}{\langle \nu,x_{1}\rangle}\big)\big(r\frac{\partial^{2}V}{\partial r\partial y_{l+1}}
+\sum\limits_{j=3}^{N}y_{j}\frac{\partial^{2}V}{\partial y_{j}\partial y_{l+1}}\big)
\Big](r_{0},y''_{0})b_{l+1,k}=O\big(\lambda_{m_k}^{-1-\epsilon}\big).
    \end{array}
  \right.
\end{equation}
Obviously, the coefficient matrix of the system
\eqref{eqs-main} is just the matrix $(A_{i,l})_{(N-1)\times (N-1)},$
where $A_{i,l},i,l=1,2,...,N-1$
are defined in \eqref{eq-matrix}.

By the assumption $(\tilde{V}')$
and the theory of solutions for
homogenous linear equations in linear Algebra, we know that the only solution
of the system \eqref{eqs-main} is $b_{1,k}=o(1),b_{l,k}=o(1)(l=3,\cdots,N).$
\end{proof}

Now we are in a position to prove Proposition \ref{thm1.2}.
\begin{proof}[\textbf{Proof of Proposition~\ref{thm1.2}}]
First we have
\begin{align}\label{eqs1.1.1}
|\eta_{k}(y)|
&\leq C\int_{\R^{N}}\frac{1}{|y-z|^{N-2}}|u^{2^{*}-2}_{m_{k}}(z)||\eta_{k}(z)|dz\cr
&\leq C\|\eta_{k}\|_{*}\int \frac{1}{|y-z|^{N-2}}|u^{2^{*}-2}_{m_{k}}(z)|
\sum_{j=1}^{m_{k}}\frac{\lambda_{m_{k}}^{\frac{N-2}{2}}}{(1+\lambda_{m_{k}}|z-x_{m_{k},j}|)^{\frac{N-2}{2}+\tau}}\cr
&\leq C\|\eta_{k}\|_{*}\sum_{j=1}^{m_{k}}\frac{\lambda_{m_{k}}^{\frac{N-2}{2}}}{(1+\lambda_{m_{k}}|y-x_{m_{k},j}|)^{\frac{N-2}{2}+\tau+\theta}},
\end{align}
for some $\theta>0.$
Hence we have
$$
\frac{|\eta_{k}(y)|}{\ds\sum_{j=1}^{m_{k}}\frac{\lambda_{m_{k}}^{\frac{N-2}{2}}}{(1+\lambda_{m_{k}}|y-x_{m_{k},j}|)^{\frac{N-2}{2}+\tau}}}
\leq C\|\eta_{k}\|_{*}\frac{\ds\sum_{j=1}^{m_{k}}\frac{\lambda_{m_{k}}^{\frac{N-2}{2}}}{(1+\lambda_{m_{k}}|y-x_{m_{k},j}|)^{\frac{N-2}{2}+\tau+\theta}}}
{\ds\sum_{j=1}^{m_{k}}\frac{\lambda_{m_{k}}^{\frac{N-2}{2}}}{(1+\lambda_{m_{k}}|y-x_{m_{k},j}|)^{\frac{N-2}{2}+\tau}}}.
$$
Since $\eta_{k}\rightarrow 0$
in $B_{R\lambda^{-1}_{m_{k}}}(x_{m_{k},j})$ and $\|\eta_{k}\|_{*}=1,$ we know that
$$
\frac{|\eta_{k}(y)|}{\ds\sum_{j=1}^{m_{k}}\frac{\lambda_{m_{k}}^{\frac{N-2}{2}}}{(1+\lambda_{m_{k}}|y-x_{m_{k},j}|)^{\frac{N-2}{2}+\tau}}}
$$
attains its maximum in $\R^{N}\backslash \cup_{j=1}^{m_{k}}B_{R\lambda^{-1}_{m_{k}}}(x_{m_{k},j}).$ Therefore
$$
\|\eta_{k}\|_{*}\leq o(1)\|\eta_{k}\|_{*}.
$$
Hence $\|\eta_{k}\|_{*}\rightarrow 0$
as $k\rightarrow \infty.$ This contradicts with $\|\eta_{k}\|_{*}=1.$
\end{proof}

\begin{remark}
If $V(y)$ is radial, then the assumption $(\tilde{V}')$
is just
$$
 \Delta V-\Big(\Delta V+\frac{1}{2}(\Delta V)'\Big)r\neq0 \,\,\,\hbox{at}\,\,\, r=r_0.
$$
\end{remark}

\begin{remark}
We want to mention that
very recently  in \cite{GMPY-21-arxiv} Guo ect. applied the local Pohozaev identities together with the Green function $G(y, x)=\frac1{(N-2)\omega_{N-1}}\frac1{|y-x|^{N-2}}$
to reprove the non-degeneracy result in \cite{GMPY-20-JFA} and get rid of the assumption
$\Delta K-(\Delta K+\frac{1}{2}(\Delta K)')r\neq 0$ at $r=r_{0}.$
So we conjecture that we may not need the assumption $(\tilde{V}')$ if we apply the similar argument as that of \cite{GMPY-21-arxiv}.
And we will return to this topic in a future work.
\end{remark}

We would like to point out that the local Pohozaev identities play a crucial role in the investigation of the non-degeneracy
of the multi-bubbling solutions. This novel idea first comes from \cite{GMPY-20-JFA}.
Also, the non-degeneracy of the solution and the uniqueness of such a solution are two very closely related problems
which are both of great interest.

 \begin{remark}\label{rem-t}
From the proof of Proposition \ref{thm1.2}, if we substitute the assumption
 $(\tilde{V}')$ for the assumption
 $(\tilde{V}),$ then only by making some minor modifications we can also prove
 the bubbling solution $u_{m}$ in Remark \ref{rem-t-8-6} is non-degenerate.
\end{remark}
It follows from Remark \ref{rem-t-8-6}
and Remark \ref{rem-t} that if the assumptions $(V)$
and $(\tilde{V})$
hold, then problem \eqref{1.3}
has a non-degenerate $m$-bubbling solution of the form
$u_{m}=Z_{\bar r_{m},\hat{y}^{*}_{m},\lambda_{m},}+\varphi_{m}=\sum\limits_{j=1}^{m}\hat{\zeta}
U_{\hat{x}_{j},\lambda_{m}}+\varphi_{m},$
where $\varphi_{m}\in H_{s}$. Moreover,  as
$m\rightarrow +\infty$,
$\lambda_{m}\in[L_{0}m^{\frac{N-2}{N-4}},L_{1}m^{\frac{N-2}{N-4}}]$,
$ (\bar r_m,\bar{y}_{m}^{*})\to (r_{0},y^{*}_{0}),$ and
$\lambda_{m}^{-\frac{N-2}{2}}\|\varphi_{m}\|_{L^{\infty}}\rightarrow
0$.

\section{Construction of a new bubble solution}\label{s3}
With the non-degenerate result obtained in section \ref{s2} at hand, we can
construct a new multi-bubbling solution for
\eqref{1.3} by the finite-dimensional reduction method.

Set $n\geq m$ be a large even integer. Recall that
\begin{align*}
\hat{x}_j=\Big(\bar{r}\cos\frac{2(j-1)\pi}{m},\bar{r}\sin\frac{2(j-1)\pi}{m},0,0,\tilde{y}^*\Big),\,
j=1,\cdots,m,\,\tilde{y}^*=(\bar{y}_5,\bar{y}_6,\cdots,\bar{y}_N),
\end{align*}
and
\begin{align*}
p_j=\Big(0,0,t\cos\frac{2(j-1)\pi}{n},t\sin\frac{2(j-1)\pi}{n},\tilde{y}^*),
\end{align*}
where $t$ is close to $r_0$ and $\tilde{y}^*$ is close to $y^{*}_{0}=(y_{0,5},y_{0,6},...,y_{0,N})\in \R^{N-4}.$

Define
\begin{equation}\label{n1}
\|u\|_{\tilde{*}}=\sup_{y\in\R^{N}}\Big(\sum_{j=1}^{n}
\frac{1}{(1+\mu_n|y-p_{n,j}|)^{\frac{N-2}{2}+\tau}}\Big)^{-1}\mu_n^{-\frac{N-2}{2}}|u(y)|
\end{equation}
and
\begin{equation}\label{n2}
\|f\|_{\tilde{*}\tilde{*}}=\sup_{y\in\R^{N}}
\Big(\sum_{j=1}^{n}\frac{1}{(1+\mu_n|y-p_{n,j}|)^{\frac{N+2}{2}+\tau}}\Big)^{-1}\mu_n^{-\frac{N+2}{2}}|f(y)|,
\end{equation}
where\,$p_{n,j}=(0,0,t_n\cos\frac{2(j-1)\pi}{n},t_n\sin\frac{2(j-1)\pi}{n},x^{*}_{n})$, $\tau=\frac{N-4}{N-2}.$

Let $u_m$ be the $m$-bubbling solutions in Remark \ref{rem-t}, where $m>0$ is a large even integer.
 Since $m$ is even, $u_{m}$ is even in $y_j,j=1,2,3,4.$

We define
\begin{align*}
X_s=\Bigr\{&u: u\in H_s, u\, \hbox{is even in}\, y_h,h=1,2,3,4,\cr
&u(y_1,y_2,t\cos\theta,t\sin\theta, y^*)=u\big(y_1,y_2,t\cos(\theta+\frac{2\pi j}{n}),t\sin(\theta+\frac{2\pi j}{n}),y^*\big)\Bigl\},
\end{align*}
where $y^*=(y_5,y_6,\cdots,y_N)$.

Denote
\begin{align*}
\mathbb{M}_j=\Bigl\{y=(y',y_3,y_4,y^*)\in \R^2\times\R^2\times\R^{N-4}:\langle\frac{(y_3,y_4)}{|(y_3,y_4)|},\frac{(p_{j3},p_{j4})}{|(p_{j3},p_{j4})|}\rangle\geq\cos\frac{\pi}{n}\Bigr\}.
\end{align*}

Assume that
\begin{align}\label{thelta}
|(t,\tilde{y}^{*})-(r_{0},y^{*}_{0})|\leq \vartheta,
\end{align}
where $\vartheta>0$ is a small constant.

 Observe that both $u_m$ and $\sum\limits_{j=1}^nU_{p_j,\mu}$ belong to $X_s$, while $u_m$ and $\sum\limits_{j=1}^nU_{p_j,\mu}$ are separated from each other. We intend to construct a solution for \eqref{1.3} of the form
\begin{align*}
u=u_m+\sum\limits_{j=1}^n\hat{\zeta}(y)U_{p_j,\mu}+\psi:=u_m+\hat{\zeta}(y)Z^{*}_{t,\tilde{y}^{*},\mu}(y)+\psi
:=u_m+Z_{t,\tilde{y}^{*},\mu}(y)+\psi,
\end{align*}
where $\psi\in X_s$ is a small perturbed term.
Recall that $Z_{p_j,\mu}=\hat{\zeta}(y)U_{p_j,\mu}.$

Define the linear operator
\begin{align}\label{eqs3.01}
Q_n\psi=-\Delta\psi+V(y)\psi-(2^*-1)\Big(u_m+\sum\limits_{j=1}^nZ_{p_j,\mu}\Big)^{2^*-2}\psi,\,\,\,\,\,\psi\in X_s.
\end{align}
Denote
\begin{align*}
D_{j,1}=\frac{\partial Z_{p_j,\mu}}{\partial \mu},\,D_{j,2}=\frac{\partial Z_{p_j,\mu}}{\partial t},\,\,\,D_{j,k}=\frac{\partial Z_{p_j,\mu}}{\partial \tilde{y}_k^*},k=5,6,\cdots,N.
\end{align*}
Let $g_n\in X_s$. Now we consider
\begin{eqnarray}\label{eqs0.1}
\begin{cases}
Q_n \psi_n= g_n+\sum\limits_{i=1}^{N-2}a_{n,i}\sum\limits_{j=1}^nZ_{p_j,\mu}^{2^*-2}D_{j,i} ,\,\,\,\cr
\psi_n\in X_s,\cr
\ds\int_{\R^N}Z_{p_j,\mu}^{2^*-2}D_{j,i}\psi_n=0,\,\,i=1,\cdots,N-2, j=1,2,\cdots,n
 \end{cases}
\end{eqnarray}
for some constants $a_{n,i}$, depending on $\psi_n$.

\begin{lemma}\label{lem3.1}
Assume that $V(y)\geq 0$ is bounded in $\R^{N}.$
Suppose that $\psi_n$ solves \eqref{eqs0.1}. If $\|g_n\|_{\tilde{*}\tilde{*}}\rightarrow0$, then $\|\psi_n\|_{\tilde{*}}\rightarrow0$.
\end{lemma}

 \begin{proof}
We   argue by contradiction. Suppose  that there exist $n\rightarrow
+\infty,\,\bar t_n\to r_{0}, \bar{y}^{*}_{n}\to  y^{*}_{0}, \mu_{n}\in
[L_{0}n^{\frac{N-2}{N-4}},L_{1}n^{\frac{N-2}{N-4}}]$ and
$\psi_{n}$ solving \eqref{eqs0.1} for
$g=g_{n},\mu=\mu_{n},\bar t=\bar t_{n},\bar{y}^{*}=\bar{y}^{*}_{n}$ with
$\|g_{n}\|_{\tilde{*}\tilde{*}}\rightarrow 0$ and $\|\psi_{n}\|_{\tilde{*}}\geq c>0.$ We
may assume that $\|\psi_{n}\|_{\tilde{*}}=1$. For simplicity, we drop the
subscript $n.$


Now we rewrite \eqref{eqs0.1}
\begin{eqnarray}\label{eq-12-3}
\begin{cases}
L_{m}\psi=(2^{*}-1)\Big[\big(u_{m}+Z_{t,\tilde{y}^{*},\mu}\big)^{2^{*}-2}-u^{2^{*}-2}_{m}\Big] \psi +g_n+\sum\limits_{i=1}^{N-2}a_{n,i}\sum\limits_{j=1}^nZ_{p_j,\mu}^{2^*-2}D_{j,i} ,\,\,\,\cr
\psi_n\in X_s,\cr
\ds\int_{\R^N}Z_{p_j,\mu}^{2^*-2}D_{j,i}\psi=0,\,\,i=1,\cdots,N-2, j=1,2,\cdots,n.
 \end{cases}
\end{eqnarray}

Then we have
\begin{eqnarray}\label{eq-12-4}
\psi(y)=\int_{\R^{N}}G_{m}(y,z)\Big\{(2^{*}-1)\Big[\big(u_{m}+Z_{t,\tilde{y}^{*},\mu}\big)^{2^{*}-2}-u^{2^{*}-2}_{m}\Big] \psi +g_n+\sum\limits_{i=1}^{N-2}a_{n,i}\sum\limits_{j=1}^nZ_{p_j,\mu}^{2^*-2}D_{j,i}\Big\},
\end{eqnarray}
where $G_{m}$ is the solution of equation \eqref{2-25-10n} in section \ref{sb-add}.

For $|y|\leq R,$ by Proposition \ref{add-propa.1} we get
\begin{eqnarray}\label{eq-12-5}
|\psi(y)|&\leq &
C\int_{\R^{N}}G_{m}(z,y)\Big|(2^{*}-1)\Big[\big(u_{m}+Z_{t,\tilde{y}^{*},\mu}\big)^{2^{*}-2}-u^{2^{*}-2}_{m}\Big] \psi +g_n+\sum\limits_{i=1}^{N-2}a_{n,i}\sum\limits_{j=1}^nZ_{p_j,\mu}^{2^*-2}D_{j,i}\Big|\cr
&\leq&\int_{\R^{N}}\frac{C}{|z-y|^{N-2}}\Big|\Big[\big(u_{m}+Z_{t,\tilde{y}^{*},\mu}\big)^{2^{*}-2}-u^{2^{*}-2}_{m}\Big] \psi +g_n+\sum\limits_{i=1}^{N-2}a_{n,i}\sum\limits_{j=1}^nZ_{p_j,\mu}^{2^*-2}D_{j,i}\Big|\cr
&\leq&\int_{\R^{N}}\frac{C}{|z-y|^{N-2}}\Big( \big(Z_{t,\tilde{y}^{*},\mu}\big)^{2^{*}-2}|\psi| +|g_n|+\big|\sum\limits_{i=1}^{N-2}a_{n,i}\sum\limits_{j=1}^nZ_{p_j,\mu}^{2^*-2}D_{j,i}\big|\Big),
\end{eqnarray}
where we use the following \textbf{claim} that
\begin{eqnarray}\label{eq-12-6}
\Big|\big(u_{m}+Z_{t,\tilde{y}^{*},\mu}\big)^{2^{*}-2}-u^{2^{*}-2}_{m}\Big|\leq
CZ_{t,\tilde{y}^{*},\mu}^{2^{*}-2}.
\end{eqnarray}
Indeed, when $u_{m}\leq Z_{t,\tilde{y}^{*},\mu},$ there holds
\begin{eqnarray}\label{eq-12-7}
\Big|\big(u_{m}+Z_{t,\tilde{y}^{*},\mu}\big)^{2^{*}-2}-u^{2^{*}-2}_{m}\Big|\leq
C\Big(u_{m}^{2^{*}-2}+Z_{t,\tilde{y}^{*},\mu}^{2^{*}-2}\Big)
\leq CZ_{t,\tilde{y}^{*},\mu}^{2^{*}-2}.
\end{eqnarray}
When $Z_{t,\tilde{y}^{*},\mu}\leq u_{m},$ applying the following estimate
$$
\big|(1+x)^{\alpha}-1\big|\leq C|x|^{\alpha},\,\,\,\text{if}\,\,\,\alpha\leq 1, \,\,|x|<1,
$$
noting that $2^{*}-2<1$ when $N\geq 7,$ one has
\begin{eqnarray}\label{eq-12-8}
\Big|\big(u_{m}+Z_{t,\tilde{y}^{*},\mu}\big)^{2^{*}-2}-u^{2^{*}-2}_{m}\Big|&=&
\Big|u_{m}^{2^{*}-2}\big[\big(1+\frac{Z_{t,\tilde{y}^{*},\mu}}{u_{m}}\big)^{2^{*}-2}-1\big]\Big|
\cr&\leq &C\big|u_{m}^{2^{*}-2}\big|\Big|\frac{Z_{t,\tilde{y}^{*},\mu}}{u_{m}}\Big|^{2^{*}-2}
=CZ_{t,\tilde{y}^{*},\mu}^{2^{*}-2}.
\end{eqnarray}
Combining \eqref{eq-12-7}
and \eqref{eq-12-8}, we know that \eqref{eq-12-6} is true.

As in \cite{WY1}, using Lemmas \ref{lemb2} and \ref{lemb3},
we can prove
\begin{equation}\label{2.5}
\begin{array}{ll}
\ds\int_{\R^N} \frac{1}{|z-y|^{N-2}}
\big(Z_{t,\tilde{y}^{*},\mu}\big)^{2^{*}-2}|\psi|dz
\leq C\|\psi\|_{*}
\ds\sum_{j=1}^{n}\frac{\mu^{\frac{N-2}{2}}}{(1+\mu|y-p_{j}|)^{\frac{N-2}{2}+\tau+\iota}}.
 \end{array}
\end{equation}
For $|y|\geq R,$  by the Poisson formula and noting that $V(y)$ is nonnegative,
from \eqref{eqs0.1} we have
\begin{eqnarray}\label{eq-12-5}
|\psi(y)|&\leq &
\int_{\R^{N}}\frac{C}{|z-y|^{N-2}}\Big(\big(Z_{t,\tilde{y}^{*},\mu}\big)^{2^{*}-2} |\psi | +|g_n|+\Big|\sum\limits_{i=1}^{N-2}a_{n,i}\sum\limits_{j=1}^nZ_{p_j,\mu}^{2^*-2}D_{j,i}\Big|\Big)\cr
&&+\int_{\R^{N}}\frac{C}{|z-y|^{N-2}}u_{m}^{2^*-2}|\psi |.
\end{eqnarray}
Now we first estimate

\[
\int_{\mathbb R^N} \frac1{|z-y|^{N-2}}u_m^{2^*-2}| \psi|\le
\|\psi\|_{*} \int_{\mathbb R^N} \frac1{|z-y|^{N-2}}u_m^{2^*-2} \sum_{j=1}^n \frac{\mu^{\frac{N-2}2}}{ ( 1+ \mu |z-p_j|)^{\frac{N-2}2+\tau}}.
\]
Let  $d_j= \frac12 |y-p_j|$. Then, from \eqref{eq-u} we have
\[
\begin{split}
&\int_{ B_{d_j}(p_j)} \frac1{|z-y|^{N-2}}u_m^{2^*-2}  \frac{\mu^{\frac{N-2}2}}{ ( 1+ \mu |z-p_j|)^{\frac{N-2}2+\tau}}\\
\le &
\frac{C}{   \mu^\tau d_j^{N-2}}\int_{ B_{d_j}(p_j)}
 \frac{1}{ |z-p_j|^{\frac{N-2}2+\tau}}\frac1{(1+|z|)^4}\\
 \le &\frac{C}{   \mu^\tau d_j^{\frac{N-2}2+\tau+2}}\le \frac{C\mu^{\frac{N-2}2}}{ ( 1+ \mu |y-p_j|)^{\frac{N-2}2+\tau}}\frac1{|y|^2},
\end{split}
\]
if $ N>6+2\tau $,

\[
\int_{ B_{d_j}(p_j)} \frac1{|z-y|^{N-2}}u_m^{2^*-2}  \frac{\mu^{\frac{N-2}2}}{ ( 1+ \mu |z-p_j|)^{\frac{N-2}2+\tau}}\\
\le
\frac{C\mu^{\frac{N-2}2}}{ ( 1+ \mu |y-p_j|)^{\frac{N-2}2+\tau}}\frac{\ln |y|}{|y|^2},
\]
if $ N=6+2\tau $,
  and
\[
\begin{split}
&\int_{ B_{d_j}(p_j)} \frac1{|z-y|^{N-2}}u_m^{2^*-2}  \frac{\mu^{\frac{N-2}2}}{ ( 1+ \mu |z-p_j|)^{\frac{N-2}2+\tau}}\\
\le&
\frac{C}{   \mu^\tau d_j^{N-2}}
 \le  \frac{C\mu^{\frac{N-2}2}}{ ( 1+ \mu |y-p_j|)^{\frac{N-2}2+\tau}}\frac1{|y|^{\frac{N-2}2-\tau}},
\end{split}
\]
if $ N<6+2\tau $.
By direct computations, we also have
\[
\begin{split}
&\int_{\mathbb R^N\setminus B_{d_j}(p_{j})} \frac1{|z-y|^{N-2}}u_m^{2^*-2}  \frac{\mu^{\frac{N-2}2}}{ ( 1+ \mu |z-p_j|)^{\frac{N-2}2+\tau}}\\
\le &
\frac{C\mu^{\frac{N-2}2}}{ ( 1+ \mu|y-p_j|)^{\frac{N-2}2+\tau}}\int_{\mathbb R^N\setminus  B_{d_j}(p_{j})} \frac1{|z-y|^{N-2}}\frac1{(1+|z|)^4}
\\
\le &\frac{C\mu^{\frac{N-2}2}}{ ( 1+ \mu |y-p_j|)^{\frac{N-2}2+\tau}}\frac1{|y|^2}.
\end{split}
\]
Hence we get that for $|y|\ge R$,
\begin{equation}\label{21-1-6-21}
\begin{split}
 \int_{\R^{N}}\frac{C}{|z-y|^{N-2}}u_{m}^{2^*-2}|\psi |\le
C\|\psi\|_{*}
\Big(\ds\sum_{j=1}^{n}\frac{\mu^{\frac{N-2}{2}}}{(1+\mu|y-p_{j}|)^{\frac{N-2}{2}+\tau}}
\frac{1}{|y|^\sigma}
\Big),
\end{split}
\end{equation}
for some $\sigma>0.$


Moreover we can obtain
\begin{equation}\label{2.6}
\begin{array}{ll}
\ds\int_{\R^N} \frac{1}{|z-y|^{N-2}}|g(z)|dz
\leq  C\|g\|_{\tilde{*}\tilde{*}}\mu^{\frac{N-2}{2}}\ds\sum_{j=1}^{n}\frac{1}{(1+\mu|y-x_{j}|)^{\frac{N-2}{2}+\tau}}
 \end{array}
\end{equation}
and
\begin{equation}\label{2.7}
\begin{array}{ll}
\ds\int_{\R^N}
\frac{1}{|z-y|^{N-2}}\Big|\ds\sum_{j=1}^{n}D_{x_{j},\mu}^{2^{*}-2}Z_{j,l}\Big|dz
\leq
C\mu^{\frac{N-2}{2}+n_{l}}\ds\sum_{j=1}^{n}\frac{1}{(1+\mu|y-p_{j}|)^{\frac{N-2}{2}+\tau}},
 \end{array}
\end{equation}
where $n_{j}=1$, $j=2,\cdots,N-2,$ and $n_{1}=-1.$

To estimate $a_{l},l=1,2,\cdots,N-2,$ multiplying \eqref{eqs0.1} by $D_{1,l}(l=1,2,\cdots,N-2)$ and integrating, we see that $a_{l}$ satisfies
\begin{equation}\label{2.8}
\begin{split}
&\ds\sum_{h=1}^{N-2}a_{h}\sum_{j=1}^{n}\int_{\R^N} Z_{p_{j},\mu}^{2^{*}-2}D_{j,h}D_{1,l}
\\
=&
\Bigl\langle \ds-\Delta \psi+V(r,y^{*})\psi
-(2^{*}-1)Z_{t,y_{0}^{*},\mu}^{2^{*}-2}
\psi,D_{1,l}\Bigr\rangle-\langle g, D_{1,l}\rangle.
 \end{split}
 \end{equation}

It follows from Lemma \ref{lemb1} that
\begin{equation}\label{2.8'}
 \begin{split}
&\big|\big\langle g, D_{1,l}\big\rangle\big|\\
\leq &C\|g\|_{\tilde{*}\tilde{*}}\ds\int_{\R^N}\frac{\mu^{\frac{N-2}{2}+n_{l}}}{(1+\mu|y-p_{1}|)^{N-2}}
\sum_{j=1}^{n}\frac{\mu^{\frac{N+2}{2}}}{(1+\mu|y-p_{j}|)^{\frac{N+2}{2}+\tau}}\\
\leq & C \mu^{n_{l}}\|g\|_{\tilde{*}\tilde{*}}\Big( C + C\sum_{j=2}^{n}\frac1{(\mu|p_j-p_1|
)^\tau}\Big)\leq C \mu^{n_{l}}\|g\|_{\tilde{*}\tilde{*}}.
 \end{split}
\end{equation}

 Similar to (2.10) in \cite{PWY-18-JFA}, we can estimate
 \begin{equation}\label{2.9.2}
\begin{array}{ll}
|  \bigl\langle  V(r,y^{*})\psi
,D_{1,l}\bigr\rangle |=O\Bigl(
\ds\frac{ \mu^{n_{l}}\|\psi\|_{\tilde{*}}}{\mu^{1+\epsilon}}\Bigr),
\end{array}
 \end{equation}
where we use the fact that for any $| (r,y^{*})-(r_0, y_0^{*})|\le 2\delta,$
\begin{equation}\label{add1}
\frac1\mu\le \frac C{1+\mu|y-p_j|}.
\end{equation}

On the other hand, direct calculation gives
\begin{equation}\label{2.9.5}
 \Big|\Bigl\langle \ds-\Delta \psi
-(2^{*}-1)(u_{m}+Z_{t,\tilde{y}^{*},\mu})^{2^{*}-2}
\psi,D_{1,l}\Bigr\rangle \Big|=O\Bigl(
\frac{\mu^{n_{l}}\|\psi\|_{\tilde{*}}}{{\mu^{1+\epsilon}}}\Bigr).
 \end{equation}

 Combining \eqref{2.8'}, \eqref{2.9.2}, \eqref{2.9.5}, we have
 \begin{equation}\label{2.14}
 \begin{array}{ll}
 \Bigl\langle\ds-\Delta \psi+V(t,y^{*})\psi
-(2^{*}-1)(u_{m}+Z_{t,\tilde{y}^{*},\mu})^{2^{*}-2}
\psi,D_{1,l}\Bigr\rangle-\langle g, D_{1,l}\rangle
=O\Bigl(\mu^{n_{l}}\big(
\frac{\|\psi\|_{\tilde{*}}}{{\mu^{1+\epsilon}}}+\|g\|_{\tilde{*}\tilde{*}}\big)\Bigr).
\end{array}
 \end{equation}

It is easy to check that

 \begin{equation}\label{2.14'}
 \sum_{j=1}^{n}\langle
D_{p_{j},\mu}^{2^{*}-2}D_{j,h},D_{1,l}\rangle
= ( \bar{c}  +o(1)) \delta_{hl}\mu^{2n_l}
 \end{equation}
 for some constant $\bar{c}>0$.

Now inserting \eqref{2.14} and \eqref{2.14'} into
 \eqref{2.8}, we find
 \begin{equation}\label{2.16}
 a_{l}=\frac{1}{\mu^{n_{l}}}\big(o(\|\psi\|_{\tilde{*}})+O(\|g\|_{\tilde{*}\tilde{*}})\big).
 \end{equation}
 So,
 \begin{equation}\label{2.11}
\|\psi\|_{\tilde{*}}\leq
\Bigg(o(1)+\|g_{n}\|_{\tilde{*}\tilde{*}}+\frac{\sum_{j=1}^{n}\frac{1}{(1+\mu|y-p_{j}|)^{\frac{N-2}{2}+\tau+\iota}}}
{\sum_{j=1}^{n}\frac{1}{(1+\mu|y-p_{j}|)^{\frac{N-2}{2}+\tau}}}\Bigg).
 \end{equation}
 Since $\|\psi\|_{\tilde{*}}=1$, we obtain from \eqref{2.11} that there is $R>0$ such that
 \begin{equation}\label{2.12}
 \|\mu^{-\frac{N-2}{2}}\psi\|_{L^\infty(B_{R/\mu}(p_{j}))}\geq a>0,
 \end{equation}
 for some $j$. But $\tilde{\psi}(y)=\mu^{-\frac{N-2}{2}}\psi(\mu^{-1}y+p_{j}))$ converges uniformly in any compact set to a solution $u$ of
\begin{equation}\label{2.13}
-\Delta  u-(2^{*}-1)U_{0,1}^{2^{*}-2}u =0,\,\,\,\text{in}\,\,\,
\R^{N},
 \end{equation}
 and $u$ is perpendicular to the kernel of \eqref{2.13}.
 Hence $u=0.$ This is a contradiction to \eqref{2.12}.
 \end{proof}

We want to construct a solution $u$ for \eqref{1.3} with
$$
u=u_m+\sum\limits_{j=1}^n\hat{\zeta} U_{p_j,\mu}+\psi,
$$
where $\psi\in X_s$ is a small perturbed term, satisfying
\begin{align*}
\int_{\R^N}Z_{p_j,\mu}^{2^*-2}D_{j,l}\psi=0,\,\,\,j=1,\cdots,n, l=1,2,\cdots, N-2.
\end{align*}
Then  $\psi$ satisfies
\begin{align*}
Q_n \psi=l_n+R_n(\psi),
\end{align*}
where
\begin{align*}
Q_n\psi=-\Delta \psi+V(y)\psi-(2^*-1)\Big(u_m+\sum\limits_{j=1}^nZ_{p_j,\mu}\Big)^{2^*-2}\psi,
\end{align*}
\begin{align*}
l_n=\Big(u_m+\sum\limits_{j=1}^nZ_{p_j,\mu}\Big)^{2^*-1}-u_m^{2^*-1}-\sum\limits_{j=1}^n\hat{\zeta} U_{p_j,\mu}^{2^*-1}
-V(y)\sum\limits_{j=1}^nZ_{p_j,\mu}
+Z^{*}_{t,\tilde{y}^{*},\mu}\Delta \hat{\zeta} +2\nabla \hat{\zeta} \nabla Z^{*}_{t,\tilde{y}^{*},\mu},
\end{align*}
and
\begin{align*}
R_n(\psi)=&\Big(u_m+\sum\limits_{j=1}^nZ_{p_j,\mu}+\psi\Big)^{2^*-1}
-\Big(u_m+\sum\limits_{j=1}^nZ_{p_j,\mu}\Big)^{2^*-1}
-(2^*-1)\Big(u_m+\sum\limits_{j=1}^nZ_{p_j,\mu}\Big)^{2^*-2}\psi.
\end{align*}

We have the following estimate for $\|l_n\|_{\tilde{*}\tilde{*}}.$
\begin{lemma}\label{lem3.2}
There exists a small $\epsilon>0$, such that
\begin{align}\label{eq-l}
\|l_n\|_{\tilde{*}\tilde{*}}\leq \frac{C }{\mu^{1+\epsilon}}.
\end{align}
\end{lemma}

\begin{proof}
First, we write
\begin{align}\label{eq-l1}
l_n=&\Big[\big(\sum\limits_{j=1}^n\hat{\zeta} U_{p_j,\mu}\big)^{2^*-1}-\sum\limits_{j=1}^n\hat{\zeta} U_{p_j,\mu}^{2^*-1}\Big]
-V(y)\sum\limits_{j=1}^n \hat{\zeta}U_{p_j,\mu}\cr
&+Z^{*}_{t,\tilde{y}^{*},\mu}\Delta \hat{\zeta} +2\nabla \hat{\zeta} \nabla Z^{*}_{t,\tilde{y}^{*},\mu}\cr
&+\Big[\big(u_m+\sum\limits_{j=1}^n\hat{\zeta} U_{p_j,\mu}\big)^{2^*-1}-u_m^{2^*-1}-\big(\sum\limits_{j=1}^n\hat{\zeta} U_{p_j,\mu}\big)^{2^*-1}\Big]\cr
:=&J_0+J_1+J_2+J_{3}+J_{4}.
\end{align}
Just by the same argument as that of Lemma 2.5 in \cite{PWY-18-JFA},
we can estimate
\begin{align}\label{eq-l1-1}
J_0+J_1+J_2+J_{3}
\leq  \frac{C}{\mu^{1+\epsilon}}
\sum_{j=1}^{n}\frac{\mu^{\frac{N+2}{2}}}{(1+\mu|y-p_{j}|)^{\frac{N+2}{2}+\tau}}.
\end{align}

Now we estimate $J_{4}.$
Using the assumed symmetry, we just need to estimate $J_4$ in $\mathbb{M}_1.$

Denote $\mathbb{B}_{1}=\mathbb{M}_1\bigcap B_{\mu^{-\frac{1}{2}}}(p_1).$
Note that, it holds $U_{p_1,\mu}\geq c_0$ in $\mathbb{B}_{1}.$

When $y\in \mathbb{B}_{1},$ applying the following formula
$$
(1+t)^{p}=1+pt+O(t^{p}),\,\,\,\text{for}\,\,t\geq 0\,\,\text{and}\,\,p\in (1,2],
$$
we have
\begin{align}\label{1}
|J_4|&\leq C U_{p_1,\mu}^{2^*-2}\Big(u_m+\sum\limits_{j=2}^nU_{p_j,\mu}\Big)+
\Big(\sum_{j=2}^{n}U_{p_{j},\mu}\Big)^{2^{*}-1}+\tilde{J}_4
:=J_{4,1}+J_{4,2}+\tilde{J}_4,
\end{align}
where $\tilde{J}_4\leq C$ in $\mathbb{B}_{1}.$

For $y\in \mathbb{M}_1$, by Lemma \ref{lemb1} we obtain
\begin{align}\label{esj42}
J_{4,2}\leq& C\mu^{\frac{N+2}{2}}
\Big(\sum\limits_{j=2}^n\frac{1}{(1+\mu|y-p_j|)^{N-2}}\Big)^{\frac{N+2}{N-2}}\cr
\leq&C\mu^{\frac{N+2}{2}}\Big(\sum\limits_{j=2}^n\frac{1}{(1+\mu|y-p_1|)^{\frac{N+2}{2}}}\frac{1}{(1+\mu|y-p_j|)^{\frac{N-2}{2}}}\Big)^{\frac{N+2}{N-2}}\cr
\leq&C\mu^{\frac{N+2}{2}}\Bigl[\sum\limits_{j=2}^n\frac{1}{(\mu|p_j-p_1|)^{\frac{N-2}{2}-\frac{N-2}{N+2}\tau}}\bigl(\frac{1}{(1+\mu|y-p_1|)^{\frac{N-2}{2}+\frac{N-2}{N+2}\tau}}\cr
&+
\frac{1}{(1+\mu|y-p_j|)^{\frac{N-2}{2}+\frac{N+2}{N-2}\tau}}\bigr)\Bigr]^{\frac{N+2}{N-2}}\cr
\leq&C\mu^{\frac{N+2}{2}}\big(\frac{n}{\mu}\big)^{\frac{N+2}{2}-\tau}\frac{1}{(1+\mu|y-p_1|)^{\frac{N+2}{2}+\tau}}\cr
\leq& \frac{C}{\mu^{\frac{N+2}{N-2}-\frac{2(N-4)}{(N-2)^2}}}\frac{\mu^{\frac{N+2}{2}}}{(1+\mu|y-p_1|)^{\frac{N+2}{2}+\tau}}\cr
\leq&\frac{C}{\mu^{1+\varepsilon}}\frac{\mu^{\frac{N+2}{2}}}{(1+\mu|y-p_1|)^{\frac{N+2}{2}+\tau}},\,\,y\in \mathbb{M}_1.
\end{align}
Noting that
\begin{align*}
\frac{\mu^{\frac{N+2}{2}}}{1+\mu|y-p_1|^{\frac{N+2}{2}+\tau}}\geq
\frac{\mu^{\frac{N+2}{2}}}{(1+\mu^{\frac{1}{2}})^{\frac{N+2}{2}+\tau}}\geq c_0\mu^{\frac{N+2}{4}-\frac{\tau}{2}},
\,\,y\in \mathbb{B}_{1},
\end{align*}
we have
\begin{align}\label{2}
c_0\leq \frac{\mu^{\frac{N+2}{2}}}{(1+\mu|y-p_1|)^{\frac{N+2}{2}+\tau}}\mu^{-(\frac{N+2}{4}-\frac{\tau}{2})}\leq \frac{C}{\mu^{1+\varepsilon}}\frac{\mu^{\frac{N+2}{2}}}{(1+\mu|y-p_1|)^{\frac{N+2}{2}+\tau}},
\end{align}
which implies that
\begin{align}\label{3}
|\tilde{J}_4|\leq \frac{C}{\mu^{1+\varepsilon}}\frac{\mu^{\frac{N+2}{2}}}{(1+\mu|y-p_1|)^{\frac{N+2}{2}+\tau}},\,\,y\in \mathbb{B}_{1}.
\end{align}
On the other hand, when $y\in \mathbb{B}_{1},$ there holds
\begin{align*}
J_{4,1}&\leq \Big|U_{p_1,\mu}^{2^*-2}(u_m+\sum\limits_{j=2}^nU_{p_j,\mu})\Big|
\leq C U_{p_1,\mu}^{2^*-2}+CU_{p_1,\mu}^{2^*-2}\sum\limits_{j=2}^nU_{p_j,\mu}\\
&:=J_{4,1,1}+J_{4,1,2}.
\end{align*}
Similar to Lemma 2.5 in \cite{PWY-18-JFA}, we can prove
\begin{align}\label{es-3}
|J_{4,1,2}|\leq \frac{C}{\mu^{1+\varepsilon}}\frac{\mu^{\frac{N+2}{2}}}
{(1+\mu|y-p_1|)^{\frac{N+2}{2}+\tau}},
\,\,\,y\in \mathbb{B}_{1}.
\end{align}
Moreover, if $N\geq 5$ and $y\in \mathbb{B}_{1},$
\begin{align*}
(1+\mu|y-p_1|)^{\frac{N+2}{2}+\tau-4}\leq C\mu^{\frac{1}{2}(\frac{N+2}{2}+\tau-4)},
\end{align*}
noting that $\frac{N+2}{4}-\frac{\tau}{2}>1,$ then we have
\begin{align}\label{es-4}
J_{4,1,1}
\leq &\frac{\mu^{\frac{N+2}{2}}}{(1+\mu|y-p_1|)^{\frac{N+2}{2}+\tau}}\frac{\mu^{2-\frac{N+2}{2}}}{(1+\mu|y-p_1|)^{4-(\frac{N+2}{2}+\tau)}}\cr
\leq& C\mu^{-(\frac{N+2}{4}-\frac{\tau}{2})}\frac{\mu^{\frac{N+2}{2}}}{(1+\mu|y-p_1|)^{\frac{N+2}{2}+\tau}}\cr
\leq&\frac{C}{\mu^{1+\varepsilon}}\frac{\mu^{\frac{N+2}{2}}}{(1+\mu|y-p_1|)^{\frac{N+2}{2}+\tau}}.
\end{align}
Therefore, it follows from \eqref{1} to \eqref{es-4} that
\begin{align}\label{esj4-sum1}
J_{4}\leq \frac{C}{\mu^{1+\epsilon}}\frac{\mu^{\frac{N+2}{2}}}{(1+\mu|y-p_1|)^{\frac{N+2}{2}+\tau}},\,y\in \mathbb{B}_{1}.
\end{align}
On the other hand, noting that in $\mathbb{M}_{1}\setminus \mathbb{B}_{1},$ it holds $U_{p_1,\mu}\leq C$, therefore,
\begin{align}\label{j4-2}
|J_4|\leq
C \sum\limits_{j=1}^nZ_{p_j,\mu}
+C\Big(\sum\limits_{j=1}^nU_{p_j,\mu}\Big)^{2^*-1}=\hat{J}_{4,1}+\hat{J}_{4,2}.
\end{align}
Observe that
\begin{align*}
\hat{J}_{4,1}=\sum\limits_{j=1}^n\hat{\zeta}(y)U_{p_j,\mu}\leq \hat{\zeta}(y)U_{p_1,\mu}+\sum\limits_{j=2}^n\hat{\zeta}(y)U_{p_j,\mu}.
\end{align*}
We have
\begin{align}\label{w1}
\hat{\zeta}U_{p_1,\mu}\leq& \frac{\hat{\zeta}(y)\mu^{\frac{N-2}{2}}}{(1+\mu|y-p_1|)^{N-2}}\leq \frac{\hat{\zeta}(y)\mu^{\frac{N+2}{2}}}{\mu^2(1+\mu|y-p_1|)^{N-2}}\cr
\leq&\frac{C}{\mu^{1+\epsilon}}\frac{\hat{\zeta}(y)}{\mu^{1-\epsilon}(1+\mu|y-p_1|)^{N-2}}\leq \frac{C}{\mu^{1+\epsilon}}\frac{\mu^{\frac{N+2}{2}}}{(1+\mu|y-p_1|)^{N-1-\epsilon}}\cr
\leq& \frac{C}{\mu^{1+\epsilon}}\frac{\mu^{\frac{N+2}{2}}}{(1+\mu|y-p_1|)^{\frac{N+2}{2}+\tau}},
\end{align}
since $N-1-\epsilon>\frac{N+2}{2}+\tau.$
Also for $y\in \mathbb{M}_{1}\setminus \mathbb{B}_{1}$ we can estimate
\begin{align}\label{w2}
&\sum\limits_{j=2}^n\hat{\zeta}(y)U_{p_j,\mu}
\leq C\sum\limits_{j=2}^n\frac{\hat{\zeta}(y)\mu^{\frac{N-2}{2}}}{(1+\mu|y-p_j|)^{N-2}}\cr
\leq&\frac{C}{\mu^{1+\epsilon}} \sum\limits_{j=2}^n\frac{\hat{\zeta}(y)\mu^{\frac{N+2}{2}}}{\mu^{1-\epsilon}(1+\mu|y-p_j|)^{N-2}}
\leq \frac{C}{\mu^{1+\epsilon}}
\frac{ n\mu^{\frac{N+2}{2}}}{\mu^{1-\epsilon}(1+\mu|y-p_1|)^{N-2}}
\cr
\leq &\frac{C}{\mu^{1+\epsilon}}\frac{ \mu^{\frac{N+2}{2}}}{\mu^{1-\tau-\epsilon}(1+\mu|y-p_1|)^{N-2}}
\leq\frac{C}{\mu^{1+\epsilon}}
\frac{\mu^{\frac{N+2}{2}}}{(1+\mu|y-p_1|)^{N-1-\tau-\epsilon}}
\cr
\leq& \frac{C}{\mu^{1+\epsilon}}\frac{\mu^{\frac{N+2}{2}}}{(1+\mu|y-p_1|)^{\frac{N+2}{2}+\tau}},
\end{align}
since $N-1-\tau-\epsilon\geq \frac{N+2}{2}+\tau.$


Finally, we have
\begin{align*}
|\hat{J}_{4,2}|\leq
C U_{p_1,\mu}^{2^*-1}
+C\Big(\sum\limits_{j=2}^nU_{p_j,\mu}\Big)^{2^*-1}=\hat{J}_{4,2,1}+\hat{J}_{4,2,2}.
\end{align*}

And
from \eqref{esj42}, we have
\begin{align}\label{esj22}
|\hat{J}_{4,2,2}|
\leq C|J_{4,2}|\leq\frac{C}{\mu^{1+\epsilon}}\frac{\mu^{\frac{N+2}{2}}}{(1+\mu|y-p_1|)^{\frac{N+2}{2}+\tau}}\,\,\,\,y\in \mathbb{M}_{1}\setminus \mathbb{B}_{1}.
\end{align}
For $y\in \mathbb{M}_{1}\setminus \mathbb{B}_{1},$
 $|y-p_1|\geq\mu^{-\frac{1}{2}}$ and $\mu|y-p_1|\geq \mu^{\frac{1}{2}},$
 then from $\frac{1}{2}(\frac{N+2}{2}-\tau)=\frac{N+2}{4}-\frac{N-4}{2(N-2)}>1,$
we have
\begin{align}\label{esj2-1}
\hat{J}_{4,2,1}\leq &C\frac{\mu^{\frac{N+2}{2}}}{(1+\mu|y-p_1|)^{N+2}}\leq C\frac{\mu^{\frac{N+2}{2}}}{(1+\mu|y-p_1|)^{\frac{N+2}{2}+\tau}}\frac{1}{(1+\mu|y-p_1|)^{\frac{N+2}{2}-\tau}}\cr
\leq&\frac{\mu^{\frac{N+2}{2}}}{(1+\mu|y-p_1|)^{\frac{N+2}{2}+\tau}}
\frac{1}{\mu^{\frac{1}{2}(\frac{N+2}{2}-\tau)}}\\
\leq &\frac{C\mu^{\frac{N+2}{2}}}{\mu^{1+\epsilon}(1+\mu|y-p_1|)^{\frac{N+2}{2}+\tau}},\,\,\,y\in \mathbb{M}_{1}\setminus \mathbb{B}_{1}.
\end{align}

From \eqref{j4-2} to \eqref{esj2-1}, we obtain
\begin{align}\label{esj4-sum2}
J_{4}\leq \frac{C}{\mu^{1+\epsilon}}
\frac{\mu^{\frac{N+2}{2}}}{(1+\mu|y-p_1|)^{\frac{N+2}{2}+\tau}},\,\,\,y\in \mathbb{M}_{1}\setminus \mathbb{B}_{1}.
\end{align}

Combining \eqref{esj4-sum1} and \eqref{esj4-sum2},  applying the symmetry we have
\begin{align}\label{j4-norm}
J_{4}\leq \frac{C}{\mu^{1+\epsilon}}\sum_{j=1}^{n}
\frac{\mu^{\frac{N+2}{2}}}{(1+\mu|y-p_j|)^{\frac{N+2}{2}+\tau}}.
\end{align}

By \eqref{eq-l1}, \eqref{eq-l1-1} and \eqref{j4-norm}, we obtain
$$
\|l_n\|_{\tilde{*}\tilde{*}}\leq \frac{C }{\mu^{1+\epsilon}}.
$$

\end{proof}

We also need the following estimate.
\begin{lemma}\label{lem3.3}
If $N\geq 5,$ then there holds
\begin{align*}
\|R_n(\psi)\|_{\tilde{*}\tilde{*}}\leq C \|\psi\|_{\tilde{*}}^{\min\{2^*-1,2\}}.
\end{align*}
\end{lemma}

\begin{proof}
Since it can be proved by the same argument as that of Lemma 2.4 in \cite{PWY-18-JFA},
here we omit its proof.
\end{proof}

We define
\begin{align*}
I(u)=\frac{1}{2}\int_{\R^N}\big(|\nabla u|^2+V(y)u^2\big)dy-\frac{1}{2^*}\int_{\R^N}|u|^{2^*}dy.
\end{align*}
Set
\begin{align*}
F(t,\tilde{y}^{*},\mu)=I\Big(u_m+\sum\limits_{j=1}^nZ_{p_j,\mu}+\psi_n\Big).
\end{align*}
To get a solution of the form $u_m+\sum\limits_{j=1}^nZ_{p_j,\mu}+\psi_n,$
we only need to find a critical point for $F(t,\tilde{y}^{*},\mu)$ in
$B_{\vartheta}(r_{0},y^{*}_{0})\times [C_{1}n^{\frac{N-2}{N-4}},C_{2}n^{\frac{N-2}{N-4}}], $
where $\vartheta>0$ is small, $C_1,\,C_2$ are different constants.

Now we will prove Theorem \ref{thm1.3}.
\begin{proof}[\textbf{Proof of Theorem~\ref{thm1.3}}]
By direct computation, we have
\begin{align}\label{eq-f}
F(t,\tilde{y}^{*},\mu)=
I\Big(u_m+\sum\limits_{j=1}^nZ_{p_j,\mu}\Big)
+nO\Big(\frac{1}{\mu^{2+\epsilon}}\Big).
\end{align}
On the other hand, we get
\begin{align}\label{eq-i1}
I\Big(u_m+\sum\limits_{j=1}^nZ_{p_j,\mu}\Big)
=&I\Big(\sum\limits_{j=1}^nZ_{p_j,\mu}\Big)+I(u_m)+\frac{1}{2}\int_{\R^N}\sum\limits_{j=1}^nu^{2^*-1}_mZ_{p_j,\mu}\cr
&-\frac{1}{2^*}\int_{\R^N}
\Big(\big(u_m+\sum\limits_{j=1}^nZ_{p_j,\mu}\big)^{2^*}-u_m^{2^*}-\big(\sum\limits_{j=1}^nZ_{p_j,\mu}\big)^{2^*}\Big).
\end{align}
It is not difficult to check
\begin{align*}
\int_{\R^N}u_m^{2^*-1}Z_{p_j,\mu}=O\Big(\frac{1}{\mu^{\frac{N-2}{2}}}\Big).
\end{align*}
For $y\in\R^N\setminus \bigcup_{j=1}^n\mathbb{B}_j,$ we have
\begin{align*}
&\Big|\big(u_m+\sum\limits_{j=1}^nZ_{p_j,\mu}\big)^{2^*}-u_m^{2^*}-
\big(\sum\limits_{j=1}^nZ_{p_j,\mu}\big)^{2^*}\Big|\cr
&\leq C u_m^{2^*-1}\sum\limits_{j=1}^nZ_{p_j,\mu}+\Big(\sum\limits_{j=1}^nZ_{p_j,\mu}\Big)^{2^*}\leq C u_m^{2^*-1}\sum\limits_{j=1}^nZ_{p_j,\mu}.
\end{align*}
Hence we obtain
\begin{align*}
&\int_{\R^N\setminus \bigcup_{j=1}^n\mathbb{B}_j}\Big|\big(u_m+\sum\limits_{j=1}^nZ_{p_j,\mu}\big)^{2^*}-u_m^{2^*}-\big(\sum\limits_{j=1}^nZ_{p_j,\mu}\big)^{2^*}\Big|\cr
&\leq C\int_{\R^N}u_m^{2^*-1}\sum\limits_{j=1}^nZ_{p_j,\mu}
=O\Big(\frac{n}{\mu^{\frac{N-2}{2}}}\Big).
\end{align*}
By symmetry, we have
\begin{align*}
&\int_{\bigcup_{j=1}^n\mathbb{B}_j}\Big|\big(u_m+\sum\limits_{j=1}^nZ_{p_j,\mu}\big)^{2^*}-u_m^{2^*}-\big(\sum\limits_{j=1}^nZ_{p_j,\mu}\big)^{2^*}\Big|\cr
&= n\int_{\mathbb{B}_1}
\Big|(u_m+\sum\limits_{j=1}^nZ_{p_j,\mu})^{2^*}-u_m^{2^*}-(\sum\limits_{j=1}^nZ_{p_j,\mu})^{2^*}\Big|.
\end{align*}
There holds
\begin{align*}
\int_{\mathbb{B}_1}u_m^{2^*}=O\Big(\frac{1}{\mu^{\frac{N}{2}}}\Big),
\end{align*}
and
\begin{align*}
&\int_{\mathbb{B}_1}\Big|\big(u_m+\sum\limits_{j=1}^nZ_{p_j,\mu}\big)^{2^*}
-\big(\sum\limits_{j=1}^nZ_{p_j,\mu}\big)^{2^*}\Big|\cr
&\leq C\int_{\mathbb{B}_1}\Big(\sum\limits_{j=1}^nZ_{p_j,\mu}\Big)^{2^*-1}\cr
&\leq C\int_{\mathbb{B}_1}
\Big(U^{2^{*}-1}_{p_{1},\mu}+\frac{\mu^{\frac{N+2}{2}}}{(1+\mu|y-p_{1}|)^{(2^{*}-1)(N-2)(1-\tau_{1})}}\Big)
\leq \frac{C}{\mu^{\frac{N-2}{2}}},
\end{align*}
where
$\tau_{1}=\frac{N-4}{(N-2)^{2}}.$

Hence we have proved
\begin{align}\label{eq-i}
I\Big(u_m+\sum\limits_{j=1}^nZ_{p_j,\mu}\Big)
=I\Big(\sum\limits_{j=1}^nZ_{p_j,\mu}\Big)+I(u_m)+O\Big(\frac{n}{\mu^{\frac{N-2}{2}}}\Big).
\end{align}

Moreover, by direct computation we can obtain
\begin{align}\label{eq-Un}
I\Big(\sum\limits_{j=1}^nZ_{p_j,\mu}\Big)
=n\Big(A_{1}+A_{2}\frac{V(t,\tilde{y}^{*})}{\mu^{2}}
-\sum_{j=2}^{n}\frac{A_{3}}{\mu^{N-2}|p_{j}-p_{1}|^{N-2}}
+O\big(\frac{1}{\mu^{2+\epsilon}}\big)\Big),
\end{align}
where $A_{1}=\big(\ds\frac{1}{2}-\frac{1}{2^{*}}\big)\int_{\R^{N}}U^{2^{*}}_{0,1}$ and $A_{i}(i=2,3)$
are some positive constants.

It follows from \eqref{eq-f}, \eqref{eq-i} and \eqref{eq-Un} that
\begin{align}\label{eq-fend}
&F(t,\tilde{y}^{*},\mu)=
I\Big(\sum\limits_{j=1}^nZ_{p_j,\mu}\Big)+I(u_{m})+
nO\Big(\frac{1}{\mu^{2+\epsilon}}\Big)\cr
&=I(u_{m})+nA_{1}
+n\Big(\frac{A_{2}}{\mu^{2}}V(t,\tilde{y}^{*})
-\sum_{j=2}^{n}\frac{A_{3}}{\mu^{N-2}|p_{1}-p_{j}|^{N-2}}\Big)
+O\Big(\frac{n}{\mu^{2+\epsilon}}\Big),
\end{align}
where $A_{i}(i=1,2,3)$ are the same as those of \eqref{eq-Un}.

Now in order to find a critical point for $F(t,\tilde{y}^{*},\mu)$, we only need to
continue exactly as section 3 in \cite{PWY-18-JFA}. One can also see section 3 in \cite{PWW-19-JDE}.
Here we omit the detailed process of its proof.

\end{proof}

\appendix
\section{{Some Pohozaev identities}}\label{sa}
Set
\begin{align}\label{eqs2.1}
-\Delta u+V(|y'|,y'')u=u^{2^*-1},
\end{align}
and
\begin{align}\label{eqs2.2}
-\Delta \eta+V(|y'|,y'')\eta=(2^*-1)u^{2^*-2}\eta.
\end{align}

Then standard arguments give $u,\, \eta\in L^\infty(\mathbb R^N)$ and

\begin{align}\label{eq-u}
 |u(y)|,\; \; |\eta(y)|\le \frac C{(1+ |y|)^{N-2}}.
\end{align}

Suppose that $\Omega$ is a smooth domain in $\R^N$.

 We have the following identities which are used in section \ref{s2}
 by proving the non-degeneracy of the multi-bubbling solutions obtained in \cite{PWY-18-JFA}.
\begin{lemma}\label{lem2.1}
There holds
\begin{align}\label{eqs2.3}
&-\int_{\partial\Omega}\frac{\partial u}{\partial\nu}\frac{\partial\eta}{\partial y_i}-\int_{\partial\Omega}\frac{\partial\eta}{\partial\nu}\frac{\partial u}{\partial y_i}
+\int_{\partial\Omega}\langle\nabla u,\nabla \eta\rangle\nu_i+\int_{\partial\Omega}Vu\eta\nu_i
-\int_{\partial\Omega}u^{2^*-1}\eta\nu_i\cr
=&\int_{\Omega}\frac{\partial V}{\partial y_i}u\eta,
\end{align}
and
\begin{align}\label{eqs2.4}
&\int_{\Omega}u\eta\langle\nabla V,y-x_0\rangle+2\int_{\Omega}V\eta u
=-\int_{\partial\Omega}u^{2^*-1}\eta\langle\nu,y-x_0\rangle
-\int_{\partial\Omega}\frac{\partial u}{\partial\nu}\langle\nabla\eta,y-x_0\rangle\cr
&-\int_{\partial\Omega}\frac{\partial\eta}{\partial\nu}\langle\nabla u,y-x_0\rangle
+\int_{\partial\Omega}\langle\nabla u,\nabla\eta\rangle\langle\nu,y-x_0\rangle+\int_{\partial\Omega}Vu\eta\langle\nu,y-x_0\rangle\cr
&+\frac{2-N}{2}\int_{\partial\Omega}\eta\frac{\partial u}{\partial\nu}+\frac{2-N}{2}\int_{\partial\Omega}u\frac{\partial\eta}{\partial\nu}.
\end{align}
\end{lemma}
\begin{proof}
\textbf{Proof of \eqref{eqs2.3}.}
First we have
\begin{align*}
\int_{\Omega}(-\Delta u+Vu)\frac{\partial \eta}{\partial y_i}=\int_{\Omega}u^{2^*-1}\frac{\partial \eta}{\partial y_i},
\end{align*}
and
\begin{align*}
\int_{\Omega}(-\Delta\eta+V\eta)\frac{\partial u}{\partial y_i}=\int_{\Omega}(2^*-1)u^{2^*-2}\eta\frac{\partial u}{\partial y_i},
\end{align*}
which implies that
\begin{align}\label{eqs2.5}
\int_{\Omega}\Big(-\Delta u\frac{\partial\eta}{\partial y_i}+(-\Delta\eta)\frac{\partial u}{\partial y_i}+V u\frac{\partial\eta}{\partial y_i}+V\eta\frac{\partial u}{\partial y_i}\Big)
=\int_{\Omega}\Big(u^{2^*-1}\frac{\partial\eta}{\partial y_i}+(2^*-1)u^{2^*-2}\eta\frac{\partial u}{\partial y_i}\Big).
\end{align}
It is easy to check that
\begin{align}\label{eqs2.6}
\int_{\Omega}\Big(u^{2^*-1}\frac{\partial \eta}{\partial y_i}+(2^*-1)u^{2^*-2}\eta\frac{\partial u}{\partial y_i}\Big)=\int_{\Omega}\frac{\partial (u^{2^*-1}\eta)}{\partial y_i}
=\int_{\partial\Omega}u^{2^*-1}\eta\nu_i.
\end{align}
Moreover, similar to (2.7) in \cite{GMPY-20-JFA}, we have
\begin{align}\label{eqs2.7}
\int_{\Omega}\Big(-\Delta u\frac{\partial\eta}{\partial y_i}+(-\Delta\eta)\frac{\partial u}{\partial y_i} \Big)=-\int_{\partial\Omega}\frac{\partial u}{\partial \nu}\frac{\partial\eta}{\partial y_i}-\int_{\partial\Omega}\frac{\partial\eta}{\partial \nu}\frac{\partial u}{\partial y_i}+\int_{\partial\Omega}\langle\nabla u,\nabla\eta\rangle\nu_i,
\end{align}
and
\begin{align}\label{eqs2.8}
\int_{\Omega}\Big(Vu\frac{\partial\eta}{\partial y_i}+V\eta\frac{\partial u}{\partial y_i}\Big)=\int_{\Omega}V \frac{\partial}{\partial y_i}(u\eta)
=\int_{\partial\Omega}Vu\eta\nu_i-\int_{\Omega}u\eta\frac{\partial V}{\partial y_i}.
\end{align}
It follows from \eqref{eqs2.5} to \eqref{eqs2.8} that \eqref{eqs2.3} holds.

\textbf{Proof of \eqref{eqs2.4}.}
It is easy to check that
\begin{align}\label{*}
&\int_{\Omega}\big((-\Delta u+V u)\langle\nabla\eta,y-x_0\rangle+(-\Delta\eta+V\eta)\langle\nabla u,y-x_0\rangle\big)\cr
=&\int_{\Omega}\big(u^{2^*-1}\langle\nabla\eta,y-x_0\rangle
+(2^*-1)u^{2^*-2}\eta\langle\nabla u,y-x_0\rangle\big).
\end{align}
We find that
\begin{align}\label{2.4.2}
&\int_{\Omega}\big(u^{2^*-1}\langle\nabla\eta,y-x_0\rangle
+(2^*-1)u^{2^*-2}\eta\langle\nabla u,y-x_0\rangle\big)\cr
=&\int_{\Omega}\langle\nabla(u^{2^*-1}\eta),y-x_0\rangle
=\int_{\partial\Omega}u^{2^*-1}\eta\langle\nu,y-x_0\rangle-N\int_{\Omega}u^{2^*-1}\eta.
\end{align}
Also similar to (2.10) in \cite{GMPY-20-JFA}, we have
\begin{align}\label{2.4.3}
&\int_{\Omega}\big(-\Delta u\langle\nabla\eta,y-x_0\rangle+(-\Delta\eta)\langle\nabla u,y-x_0\rangle\big)\cr
=&-\int_{\partial\Omega}\frac{\partial u}{\partial\nu}\langle\nabla\eta,y-x_0\rangle-\int_{\partial\Omega}\frac{\partial\eta}{\partial\nu}\langle\nabla u,y-x_0\rangle\cr
&+\int_{\partial\Omega}\langle\nabla u,\nabla\eta\rangle\langle\nu,y-x_0\rangle
+(2-N)\int_{\Omega}\langle\nabla u,\nabla\eta\rangle.
\end{align}
On the other hand, there holds
\begin{align}\label{2.4.4}
2^*\int_{\Omega}u^{2^*-1}\eta=&\int_{\Omega}\big((-\Delta u \eta+u(-\Delta\eta)+V u\eta+V\eta u\big)\cr
=&2\int_{\Omega}\langle\nabla u,\nabla\eta\rangle-\int_{\partial\Omega}\eta\frac{\partial u}{\partial\nu}-\int_{\partial\Omega}u\frac{\partial\eta}{\partial\nu}+2\int_{\Omega}Vu\eta,
\end{align}
which yields
\begin{align}\label{2.4.5}
\int_{\Omega}\langle\nabla u,\nabla\eta\rangle=\frac{2^*}{2}\int_{\Omega}u^{2^*-1}\eta+\frac{1}{2}\int_{\partial\Omega}\eta\frac{\partial u}{\partial\nu}+\frac{1}{2}\int_{\partial\Omega}u\frac{\partial\eta}{\partial\nu}-\int_{\Omega}Vu\eta.
\end{align}
Moreover, we obtain
\begin{align}\label{2.4.6}
&\int_{\Omega}\big(V u\langle\nabla\eta,y-x_0\rangle+V\eta\langle\nabla u,y-x_0\rangle\big)
=\int_{\Omega}V\langle\nabla (u\eta),y-x_0\rangle\cr
&=\int_{\partial\Omega}Vu\eta\langle\nu,y-x_0\rangle
-\int_{\Omega}u\eta\langle\nabla V,y-x_0\rangle-N\int_{\Omega}Vu\eta.
\end{align}
Therefore, from \eqref{*} to \eqref{2.4.6} we know that \eqref{eqs2.4} holds.
\end{proof}

\section{the Green's functions}\label{sb-add}

 In this section, we mainly study the Green's function of $L_m$(see the definition of \eqref{eq-12-1}).

 For any function $g$ defined in $\mathbb R^N$, we define its corresponding function $g^\star\in H_s$ as follows.

First we define $\textbf{A}_j$ as

 \[
 \textbf{A}_j z= \Bigl( r \cos\big(\theta+\frac{2j \pi}m\big),  r \sin\big(\theta+\frac{2j \pi}m\big), z''\Big),\quad j=1, \cdots, m,
 \]
 where $z= (z', z'')\in \mathbb R^N$, $z'= (r\cos\theta, r\sin \theta)\in \mathbb R^2$, $z''\in \mathbb R^{N-2}$, while

 \[
 \textbf{B}_i z
 =\begin{cases}
\bigl( z_1,\cdots, z_{i-1}, -z_i, z_{i+1},\cdots, z_N),\quad & i=2,\vspace{0.12cm} \\
\bigl( z_1,\cdots, z_{i-1}, z_i, z_{i+1},\cdots, z_N),\quad &i=1,3,4,\cdot\cdot\cdot,N.
\end{cases}
 \]
 Let

 \[
 \hat{g}(y)=\frac1 m\sum_{j=1}^m g(\textbf{A}_j y),
 \]
 and

 \[
 g^\star(y) = \frac1{(N-1)}\sum_{i=2}^N \frac12 \bigl( \hat{g} (y)+ \hat{g} (\textbf{B}_i y)\bigr).
 \]
 Then $g^\star\in H_s.$

Noting that $\delta_x$ is not in $H_s,$ we consider

\begin{equation}\label{2-25-10n}
 L_m u =  \delta_x^\star,  \quad u\in H_s.
 \end{equation}
 The solution of \eqref{2-25-10n} is denoted as $G_m(y, x)$. We want to point out that

\[
  \delta_x^\star =\frac1 {N-1}\sum_{i=2}^N \frac12 \Bigl( \frac1 m\sum_{j=1}^m \delta_{\textbf{A}_j x}+ \frac1 m\sum_{j=1}^m \delta_{\textbf{B}_i \textbf{A}_j x}\Bigr).
 \]

 \begin{proposition}\label{add-propa.1}
Assume that $V(y)\geq 0$ is bounded in $\R^{N}.$
 The solution $G_m(y, x)$ of \eqref{2-25-10n} satisfies

 \[
 |G_m(y, x)|\le  \frac1 {N-1}\sum_{i=2}^N \frac12 \Bigl( \frac1 m\sum_{j=1}^m \frac{C}{|y- \textbf{A}_j x|^{N-2}}+ \frac1 m\sum_{j=1}^m
 \frac{C}{|y- \textbf{B}_i \textbf{A}_j x|^{N-2}}\Bigr)
 \]
 for all $x\in B_R(0)$, where $R>0$ is any fixed large constant.

 \end{proposition}

 \begin{proof}

 Let $\omega_1= \frac{C_N}{ |y-x|^{N-2}}$, which satisfies $-\Delta \omega_1= \delta_x$ in $\mathbb R^N$.  Let $\omega_2$ be the solution of

 \[
 \begin{cases}
 -\Delta \omega +V(y)\omega= (2^*-1) u_m^{2^*-2} \omega_1,& \text{in}\; B_{2R}(0),\\
 \omega=0,& \text{on}\; \partial B_{2R}(0).
 \end{cases}
 \]
 Then $\omega_2\ge 0$ and

 \[
  \omega_2(y)= \int_{B_{2R}(0)}G(z, y) (2^*-1) u_m^{2^*-2} \omega_1\le C \int_{B_R(0)}\frac1{|y-z|^{N-2}}\frac1{|z-x|^{N-2}}\,dz\le \frac{C}{|y-x|^{N-4}},
 \]
 where $G(z,y)$ is the Green's function of the positive operator $-\Delta+V(y)$ in $B_{2R}(0)$ with zero boundary condition.
 We can continue this process to find $\omega_i$, which is the solution of

  \[
 \begin{cases}
 -\Delta \omega+V(y)\omega= (2^*-1) u_m^{2^*-2} \omega_{i-1},& \text{in}\; B_{2R}(0),\\
 \omega=0,& \text{on}\; \partial B_{2R}(0),
 \end{cases}
 \]
 and satisfies

 \[
 \begin{split}
 0\le  \omega_i (y)=& \int_{B_{2R}(0)}G(z, y) (2^*-1) u_m^{2^*-2} \omega_{i-1}\\
 \le &C \int_{B_{2R}(0)}\frac1{|y-z|^{N-2}}\frac1{|z-x|^{N-2(i-1)}}\,dz
 \le \frac{C}{|y-x|^{N-2i}}.
 \end{split}
 \]

 Let $i$ be large satisfying $\omega_i\in L^\infty(B_{2R}(0))$.  Define

 \[
 \omega=\sum_{l=1}^i \omega_l,
 \]
and $\upsilon= G_m(y, x)- \xi \omega^\star$, where $\xi(y)=\xi(|y|)\in C^\infty_0(B_{2R}(0))$, $\xi=1$ in $B_{\frac32 R}(0)$  and $0\le \xi\le 1$. Then we have

\begin{equation}\label{1-1-6-21}
L_m \upsilon = g,
\end{equation}
where $g\in L^\infty\cap H_s$ and $g=0$ in $\mathbb R^N\setminus B_{2R}(0)$.  Applying Proposition~\ref{thm1.2},  \eqref{1-1-6-21}
 has a solution $\upsilon\in H_s.$

We still need to prove that $|\upsilon(y)|\le \frac{C}{|y|^{N-2}}$ as $|y|\to +\infty$.

First, we claim that $|\upsilon|\le C |g|_{L^\infty(\mathbb R^N)}$.
Indeed, assume that there are $g_n\in L^\infty\cap H_s$, $\upsilon_n$ satisfying \eqref{1-1-6-21}, with $|g_{n}|_{L^\infty(\mathbb R^N)}\to 0$
and $|\upsilon_n|_{L^\infty(\mathbb R^N)}=1$.  Then, $\upsilon_n\to \upsilon$ in $C^1_{loc}(\mathbb R^N)$, which satisfies $L_m \upsilon=0$. Therefore $\upsilon=0$. On the other hand,
we have

\begin{equation}\label{2-1-6-21}
|\upsilon_n(y)| \leq C_N \int_{\mathbb R^N}\frac1{|z-y|^{N-2} }\big|(2^*-1) u_m^{2^*-2} \upsilon_n\,\big|dz+  C_N \int_{\mathbb R^N}\frac1{|z-y|^{N-2} }
|g_n|
\,dz.
\end{equation}
Hence we obtain that $|\upsilon_n(y)|\le \frac{C}{|y|^2}$ as $|y|\to +\infty,$ which contradicts to $|\upsilon_n|_{L^\infty(\mathbb R^N)}=1$.

So $\upsilon$ is bounded. Then it follows from \eqref{2-1-6-21} that  $|\upsilon(y)|\le \frac{C}{(1+|y|)^2}$. Also,
we have
\[
|\upsilon(y)|\le C\int_{\mathbb R^N}\frac1{|z-y|^{N-2} } u_m^{2^*-2}\frac{C}{(1+|y|)^2} +\frac{C}{(1+|y|)^{N-2}}\le \frac{C}{(1+|y|)^4}.
\]
Repeating this process, we can prove $|\upsilon(y)|\le \frac{C}{|y|^{N-2}}$ as $|y|\to +\infty.$
 \end{proof}

\section{{Basic estimates}}\label{s5}

For each fixed $k$ and $j$, $k\neq j$, we consider the following
function
\begin{equation}\label{b.1}
g_{k,j}(y)=\frac{1}{(1+|y-x_{j}|)^{\alpha}}\frac{1}{(1+|y-x_{k}|)^{\beta}},
\end{equation}
where $\alpha\geq 1$ and $\beta\geq 1$ are two constants.
 \begin{lem}\label{lemb1}(Lemma B.1, \cite{WY1})
 For any constants $0<\delta\leq \min\{\alpha,\beta\}$, there is a constant $C>0$, such that
 $$
 g_{k,j}(y)\leq \frac{C}{|x_{k}-x_{j}|^{\delta}}\Big(\frac{1}{(1+|y-x_{k}|)^{\alpha+\beta-\delta}}+\frac{1}{(1+|y-x_{j}|)^{\alpha+\beta-\delta}}\Big).
 $$
 \end{lem}

\begin{lem}\label{lemb2}(Lemma B.2, \cite{WY1})
For any constant $0<\delta<N-2$, there is a constant $C>0$, such
that
$$
\int \frac{1}{|y-z|^{N-2}}\frac{1}{(1+|z|)^{2+\delta}}dz\leq
\frac{C}{(1+|y|)^{\delta}}.
$$
\end{lem}

Let us recall that
$$
Z_{t,\tilde{y}^{*},\mu}(y)=\sum_{j=1}^{n}\hat{\zeta}(y) U_{p_{j},\mu}=[N(N-2)]^{\frac{N-2}{4}}\sum_{j=1}^{n}
\hat{\zeta}(y)\Big(\frac{\mu}{1+\mu^{2}|y-p_{j}|^{2}}\Big)^{\frac{N-2}{2}}.
$$

Just by the same argument as that of Lemma B.3 in \cite{PWY-18-JFA}, we can prove
\begin{lem}\label{lemb3}
Suppose that $N\geq 5.$ Then there is a small constant $\iota>0$,
such that
$$
\int
\frac{1}{|y-z|^{N-2}}Z^{\frac{4}{N-2}}_{t,\tilde{y}^{*},\mu}(z)\sum_{j=1}^{n}\frac{1}{(1+\mu|z-p_{j}|)^{\frac{N-2}{2}+\iota}}dz
\leq
\sum_{j=1}^{n}\frac{C}{(1+\mu|y-p_{j}|)^{\frac{N-2}{2}+\tau+\iota}}.
$$
\end{lem}

%

\section{{An example of the potential $V(r,y^{*})$}}\label{sb}
Here we give an example of $V(\hat{y},y^{*})$ which satisfies the assumptions $(V)$ and $(\tilde{V}).$
We define
\begin{equation*}
V(r,y^*)=
  \left\{
    \begin{array}{ll}
       r^2-4r\Big(\sum\limits_{j=5}^Ny_j\Big)+\Big(\sum\limits_{j=5}^Ny_j^2\Big)+1,\,\,\,&B_\rho(r_0,y_0^*),\\
\geq0, &\R^N\backslash B_\rho(r_0,y_0^*),
    \end{array}
  \right.
\end{equation*}
where $\rho$ is the same as that of \cite{PWY-18-JFA} and $(r_0,y_0^*)$ is defined below.
By some direct computations, we can check that
\begin{align*}
f(r,y^*):=r^2V(r,y^*)=r^4-4r^3\Big(\sum\limits_{j=5}^Ny_j\Big)+r^{2}\Big(\sum\limits_{j=5}^Ny_j^2\Big)+r^2.
\end{align*}
We have
\begin{align*}
\frac{\partial f}{\partial r}=4r^3-12r^2\Big(\sum\limits_{j=5}^Ny_j\Big)+2r\Big(\sum\limits_{j=5}^Ny_j^2\Big)+2r,
\end{align*}
and
\begin{align*}
\frac{\partial f}{\partial y_i}=-4r^3+2r^2y_i,\,\,\,\,i=5,\cdots,N.
\end{align*}
Suppose
that $\frac{\partial f}{\partial r}=0,\,\frac{\partial f}{\partial y_i}=0$, we obtain
\begin{align*}
y_i=2r,\,\,\,\hbox{for}\,i=5,\cdots,N, r_0=\sqrt{\frac{1}{8N-34}},\,\,\,\,\,\,\,y_{0,i}=2\sqrt{\frac{1}{8N-34}}(i=5,...,N).
\end{align*}
Therefore, $(r_0,y^{*}_{0})$ is a critical  point of
the function $f(r,y^*)$ and $V(r_0,y^{*}_{0})=\frac{1}{2}>0$.
Also
\begin{align*}
\frac{\partial^2f}{\partial r^2}=12r^2-24r\Big(\sum\limits_{j=5}^Ny_j\Big)+2\Big(\sum\limits_{j=5}^Ny_j^2\Big)+2,
\end{align*}
\begin{align*}
\frac{\partial^2f}{\partial r\partial y_i}=-12r^2+4ry_i,\,\,\,\,i=5,\cdots,N,
\end{align*}
and
\begin{align*}
\frac{\partial^2f}{\partial y_i^2}=2r^2(i=5,\cdots,N),\,\,\,\,\,\,\,\,\frac{\partial^2f}{\partial y_i\partial y_j}=0(i,j=5,...,N,i\neq j).
\end{align*}
By direct computation, we obtain
\begin{equation*}
B=
  \left(
    \begin{array}{cccc}
      \frac{\partial^2f}{\partial r^2}&\frac{\partial^2f}{\partial r\partial y_5}&\cdots&\frac{\partial^2f}{\partial r\partial y_N}\\
      \frac{\partial^2f}{\partial r\partial y_5}&\frac{\partial^2f}{\partial y_5\partial y_5}&\cdots&0\\
      \cdots&\cdots&\cdots&\cdots\\
      \frac{\partial^2f}{\partial r\partial y_N}&0&\cdots&\frac{\partial^2f}{\partial y_N\partial y_N}
    \end{array}
  \right)_{(N-3)\times (N-3)}.
\end{equation*}
Also by some tedious computation, we can obtain
 the eigenvalues
 of the matrix $B$ are as follows:
 \begin{align*}
|\lambda I-B|_{(r_0,y_0^*)}
=\Big[\big(\lambda-\frac{\partial^2f}{\partial r^2}\big)(\lambda-2r^2)-\sum\limits_{j=5}^N(-12r^2+4ry_{j})^2\Big]
(\lambda-2r^2)^{N-5}\Big|_{(r_0,y_0^*)}=0,
 \end{align*}
 which implies that
\begin{align*}
\Big(\lambda-\frac{\partial^2f}{\partial r^2}\mid_{(r_0,y_0^*)}\Big)(\lambda-2r_0^2)=\sum\limits_{j=5}^N(-12r_0^2+4r_0y_{0j})^2,\,\,\,\text{or}\,\,
(\lambda-2r^2)^{N-5}\Big|_{(r_0,y_0^*)}=0.
 \end{align*}
 By further computation, we can check that $\min\{\lambda_1,\lambda_2\}<0$
 and
 $$
 \lambda_3=\cdots=\lambda_{N-3}=2r_0^2>0.
 $$
Hence the assumption $(V)$ holds.

On the other hand, we recall
\begin{align*}
V(r,y)=r^2-4r\Big(\sum\limits_{j=5}^Ny_j\Big)+\Big(\sum\limits_{j=5}^Ny_j^2\Big)+1,\,\,
\text{in}\,\,B_\rho(r_0,y_0^*).
\end{align*}
We obtain in $B_\rho(r_0,y_0^*)$
\begin{align*}
\frac{\partial V}{\partial r}=2r-4\sum\limits_{j=5}^Ny_j,\,\,\,\,\,\,\,\frac{\partial V}{\partial y_i}=2y_i-\frac{4y_i}{r}\Big(\sum\limits_{j=5}^Ny_j\Big),i=1,2,3,4,
\end{align*}
\begin{align*}
\frac{\partial V}{\partial y_k}=-4r+2y_k,\,\,\,\,\,k=5,\cdots,N,
\end{align*}
\begin{align}\label{wqf1}
\frac{\partial^2V}{\partial r^2}=2,\,\,\,\,
\frac{\partial^2V}{\partial r\partial y_i}=\frac{2y_{0,i}}{r_{0}}(i=1,2,3,4),\,\,\,
\frac{\partial^2V}{\partial r\partial y_k}=-4(k=5,...,N),
\end{align}
\begin{align}\label{wqf2}
\frac{\partial^2V}{\partial y_i\partial y_i}=2-\frac{4r^2-4y_i^2}{r^3}\Big(\sum\limits_{j=5}^Ny_j\Big)(i=1,2,3,4),
\,\,\,\frac{\partial V}{\partial y_k^2}=2(k=5,\cdots,N)
\end{align}
and
\begin{align}\label{wqf3}
\frac{\partial^2V}{\partial y_k\partial y_j}=0(j\neq k, k,j=5,\cdots,N).
\end{align}
Then from \eqref{wqf1} to \eqref{wqf3}, we obtain in $B_\rho(r_0,y_0^*)$
\begin{align*}
\Delta V=&\sum\limits_{j=1}^N\frac{\partial^2V}{\partial y_i\partial y_i}=\sum\limits_{i=1}^4\Big(2-\frac{4r^2-4y_i^2}{r^3}(y_5+\cdots+y_N)\Big)+2(N-4)\cr
=&2N-\frac{12}{r}(y_5+\cdots+y_N).
\end{align*}
Hence
$$
\Delta V\Big|_{(r_{0},y^{*}_{0})}=96-22N.
$$
For $y_0^*=(y_{0,5},y_{0,6}\cdots,y_{0,N}),$ one has
\begin{align*}
\frac{\partial \Delta V}{\partial y_1}\Big|_{(r_0,y_0^*)}=\frac{12y_1}{r^3}(y_5+\cdots+y_N)\Big|_{(r_0,y_0^*)}=\frac{24y_{0,1}}{r_0^2}(N-4),
\end{align*}
\begin{align*}
\frac{\partial(\Delta V)}{\partial y_k}\Big|_{(r_0,y_0^*)}=-\frac{12}{r_0},k=5,\cdots,N.
\end{align*}

Since $y_{0,i}=2r_0(i=5,\cdots,N),$ we obtain
\begin{equation}\label{eq-matrix-2}
A_{i,l}=
  \left\{
    \begin{array}{ll}
      \Big[ 2-(\frac{24y_{0,i}(N-4)}{2r_0^2(96-22N)}+\frac{\nu_1}{\langle\nu,x_1\rangle})(34-8N)r_0],
 \text{when}\,\,i=l=1;
\vspace{0.12cm}\\
 -4,
\text{when}\,\,i=1,
l=2,3,...,N-3;
\vspace{0.12cm}\\
\cos\frac{2i\pi}{m}(-4-(\frac{-6}{r_0(96-22N)}+\frac{\nu_{i+1}}{\langle \nu,x_{1}\rangle})(34-8N)r_0),\\
 \text{when}\,\,i=2,3,...,N-3,l=1;
\vspace{0.12cm}\\
    2,
\text{when}\,\,i=l,i,l=2,3,...,N-3,
\vspace{0.12cm}\\
    0,
\text{when}\,\,i\neq l,\,i,l=2,3,...,N-3.
    \end{array}
  \right.
\end{equation}
Therefore, from \eqref{eq-matrix-2} we have
\begin{align*}
&\det(A_{i,l})_{(N-3)\times (N-3)}
=\prod_{i=2}^{N-3}A_{i,i}\big(A_{1,1}-\sum_{k=2}^{N-3}\frac{A_{k,1}A_{1,k}}{A_{k,k}}\big)\\
&=\prod_{i=2}^{N-3}2
\Big( 2-(\frac{24y_{0,i}(N-4)}{2r_0^2(8-22(N-4))}+\frac{\nu_1}{\langle\nu,x_1\rangle})(34-8N)r_0\big)
\\
&\quad\quad
+\sum_{k=2}^{N-3}2\cos\frac{2i\pi}{m}(-4-(\frac{-6}{r_0(8-22(N-4))}+\frac{\nu_{i+3}}{\langle \nu,x_{1}\rangle})(34-8N)r_0)
     \Big)\\
     &\neq 0 .
\end{align*}
Hence, the assumption $(\tilde{V})$ also holds.\\
\vspace{0.2cm}
\textbf{Acknowledgements}
The authors would like to thank the referees for their useful comments and suggestions.
 Q. He was partially supported by the fund from NSFC (No. 12061012) and the special foundation for Guangxi Ba Gui Scholars.
C. Wang was partially supported by NSFC (No.12071169)
and the Fundamental
Research Funds for the Central Universities(No. KJ02072020-0319).
Q. Wang was partially supported by NSFC (No. 12126356).

\end{document}